\documentclass[11pt]{article}
\usepackage{hyperref}
\hypersetup{backref,pdfpagemode=colorlinks=true,backref}
\usepackage{amsmath,amsxtra,amssymb,color,latexsym,epsfig,amscd,amsthm,subfigure,fancybox,epsfig}
\usepackage[mathscr]{eucal}
\usepackage{bbm}
\usepackage{float}
\usepackage{graphicx}
\usepackage{enumerate}
\usepackage{enumitem}
\usepackage{caption}
\usepackage{epsfig}
\usepackage{epstopdf}
\usepackage{cases}
\def\para#1{\vskip .4\baselineskip\noindent{\bf #1}}
\def\CC{{\mathfrak C}}

\newtheorem{thm}{Theorem}[section]
\newtheorem {asp}{Assumption}[section]
\newtheorem{lm}{Lemma}[section]

\newtheorem{deff}{Definition}[section]

\newtheorem{prop}{Proposition}[section]
\theoremstyle{definition}

\theoremstyle{remark}
\newtheorem{rem}{Remark}

\numberwithin{equation}{section}

\newcommand{\eps}{\varepsilon}

\newcommand{\G}{\mathcal{G}}

\newcommand{\M}{\mathcal{M}}
\newcommand{\F}{\mathcal{F}}

\newcommand{\E}{\mathbb{E}}

\newcommand{\PP}{\mathbb{P}}

\newcommand{\R}{\mathbb{R}}

\newcommand{\abs}[1]{\left\vert#1\right\vert}

\newcommand{\norm}[1]{\left\Vert#1\right\Vert}

\numberwithin{equation}{section}
\newcommand{\wdt}{\widetilde}

\newcommand{\bed}{\begin{displaymath}}
\newcommand{\eed}{\end{displaymath}}
\newcommand{\bea}{\bed\begin{array}{rl}}
\newcommand{\eea}{\end{array}\eed}
\newcommand{\ad}{&\!\!\!\disp}
\newcommand{\aad}{&\disp}
\newcommand{\barray}{\begin{array}{ll}}
\newcommand{\earray}{\end{array}}

\def\disp{\displaystyle}

\newcommand{\1}{\boldsymbol{1}}
\newcommand{\beq}{\begin{equation}}
\newcommand{\eeq}{\end{equation}}

\def\bar{\overline}
\def\hat{\widehat}
\def\a.s{\text{\;a.s.\;}}

\begin{document}
\title{Large Deviations Principles for Langevin Equations in Random Environment and Applications\thanks{This research
was supported in part by the National Science Foundation under grant
DMS-1710827.}}
\author{Nhu N. Nguyen\thanks{Department of Mathematics, University of Connecticut, Storrs, CT
06269,
nguyen.nhu@uconn.edu.} \and George Yin\thanks{Department of Mathematics,
University of Connecticut, Storrs, CT
06269,
gyin@uconn.edu.}}
\maketitle

\begin{abstract}
In contrast to the study of Langevin equations in a homogeneous environment in the literature, the study on Langevin equations in randomly-varying environments is relatively scarce. Almost all the existing works require random environments to have a specific formulation that is
 independent of the systems.
This paper aims to consider large deviations principles (LDPs)
of Langevin equations involving a random environment that is a process taking value in a measurable space and that is allowed to interact with the systems, without specified formulation on the random environment.
Examples and applications to statistical physics are provided.
 Our formulation of the random environment presents the main challenges and requires new approaches.
Our approach stems from
the intuition of the Smoluchowski-Kramers approximation. The techniques developed in this paper focus on the relation between the solutions of the second-order equations and the associate first-order equations.
We obtain the desired LDPs by showing a family of processes enjoy the exponential tightness and local LDPs with an appropriate
rate function.

\medskip



\medskip
\noindent{\bf Running Title.} LDPs of
Langevin equation
in Random Environment

\end{abstract}
\newpage

\section{Introduction}\label{sec:int}
Langevin equations are used to describe the motion
of particles in a fluid due to collisions with the molecules of the fluid; see e.g., \cite{Lag08}, which have been studied intensively in both mathematics and physics communities.
Take for instance,
 small particles
with strong damping \cite{Freidlin}, which are formulated as
$$
\dot x_\eps(t)= b(x_\eps(t))-\frac{\lambda}{\eps} \dot x_\eps(t)+\sqrt\eps\sigma\dot B(t).
$$
By letting $X^\eps_t=x_\eps(t/\eps)$, $w(t)=\sqrt\eps B(t/\eps)$, we obtain the so-called chemical Langevin equation
\begin{equation}\label{cLE}
\eps^2\ddot{X}^{\eps}_t=b(X^{\eps}_t)-\lambda \dot X^{\eps}_t+\sqrt\eps\sigma\dot{w}(t).
\end{equation}
Likewise,
the motion of a small particle with mass $\mu$ in the force field $b(x) +\sqrt \eps\sigma\dot w$ with a friction proportional to the velocity and
the friction coefficient
$\lambda$ is described by the following equation due to the Newton law
\begin{equation*}
\mu\ddot{x}^{\mu,\eps}_t=b(x^{\mu,\eps}_t)-\lambda\dot x^{\mu,\eps}_t+\sqrt\eps\sigma\dot{w}(t).
\end{equation*}
When $\mu=\eps^2$, this equation becomes \eqref{cLE}.
Much effort is devoted to the study of
equation \eqref{cLE} and its applications; see e.g., \cite{Freidlin,CF05,Cheng,Fre04,NY-JMP} and references therein.

While a time-homogeneous environment is usually used with the
force field $b$
not depending on any other random process,  we consider a randomly-varying
environment in this work.
We consider
$$
b(x)\rightsquigarrow b(t,x,\xi),\quad
\lambda\rightsquigarrow \lambda_\eps(t,x),\quad\sigma\rightsquigarrow \sigma_\eps(t,x),
$$
where $\xi$ indicates the random environment, which may or may not interact with the system.
As a consequence,
 equation \eqref{cLE} becomes
\begin{equation}\label{eq:F-setup}
\begin{cases}
\eps^2\ddot{X}^\eps_t=b(t,X^\eps_t,\xi_{t/\eps})-\lambda_\eps(t,X^\eps_t)\dot{X}^\eps_t
+\sqrt{\eps}\sigma_\eps(t,X^\eps_t)\dot w(t),\\
X^\eps_0=x_0\in\R^d,\quad\dot{X}^\eps_0=x_1\in \R^d,
\end{cases}
\end{equation}
where  $w(t)$ is an $m$-dimensional standard Brownian motion and $\dot w(t)$ is its formal derivative, $\xi_t$ is a random process, which may or may not depend on $X_t^\eps,w(t)$ and which takes value in a measurable space $\M$ describing how the status of environment changes randomly in time and state.
The fast scale $\xi_{t/\eps}$ is, in fact, obtained after rescaling $X^\eps_t=x_\eps(t/\eps)$.

Natural and important questions in mathematical physics and statistical mechanics include: What is the asymptotic behavior of $\{(X^\eps_t)\}_{\eps>0}$? Can we obtain an averaging principle for  $\xi$? Can such a second-order system be approximated by the corresponding overdamped system (the Smoluchowski-Kramers approximation)? What is the tail probability of the convergence?
We aim to address these questions by
obtaining a large deviations principle  (LDP for short) for  $\{X^\eps\}_{\eps>0}$ in the space of continuous functions.
Large deviations principles play an important role in
 equilibrium and non-equilibrium statistical mechanics, multi-fractals, and thermodynamic formulation of chaotic systems; see \cite{DS89,DZ98,Tou09} and references therein.

In this paper,  we first establish a LDP of \eqref{eq:F-setup}.
 The result will then be specified under different settings of $\xi_{t/\eps}$ such as diffusion processes, jump processes, and a Markov switching environment.  Applications to mathematical physics and statistical mechanics are then treated.
 The classical Smoluchowski-Kramers approximation
is dealt with in the presence of another random process interacting with the system.

  Let us assume that for each fixed time $t$ and fixed state $X_t^\eps=x$, as $\eps\to 0$, $\xi_{t/\eps}$ (which may depend on both $t$ and $X^t_\eps$) has an invariant measure denoted by $\pi_{t,x}$. 
  Intuitively,
  as $\eps\to0$, the behavior of  equation \eqref{eq:F-setup}  takes 3 phases.
 Phase 1, letting $\eps^2\to0$, \eqref{eq:F-setup} behaves as an overdamped Langevin equation
 $$
 \dot{X}^\eps_t=\frac{b(t,X^\eps_t,\xi_{t/\eps})}{\lambda_\eps(t,X^\eps_t)}
 +\sqrt{\eps}\frac{\sigma_\eps(t,X^\eps_t)}{\lambda_\eps(t,X^\eps_t)}\dot w(t).
 $$
 Phase 2, letting $\eps\to 0$, the ergodicity of $\xi_{t/\eps}$ leads to the approximation
 $$
 \dot{X}^\eps_t=\frac{\bar b(t,X^\eps_t)}{\lambda_\eps(t,X^\eps_t)}
 +\sqrt{\eps}\frac{\sigma_\eps(t,X^\eps_t)}{\lambda_\eps(t,X^\eps_t)}\dot w(t),
 $$
 where $\bar b(t,x):=\int_\M b(t,x,z)\pi_{t,x}(dz)$.
 Phase 3, letting $\sqrt\eps\to0$, the small diffusion presents less influence and the system tends to be concentrated on the averaged system
 $$
 \dot{\bar X}_t=\frac{\bar b(t,\bar X_t)}{\lambda_0(t,\bar X_t)},
 $$
 where $\lambda_0$ is a limit (as $\eps\to 0$) of sequence of functions $\{\lambda_\eps\}$.  Not only 
 our work provides a rigorous analysis for these intuitions, but also show that the tail probability of the convergence is exponentially small under appropriate conditions.

 \para{Related works.}
In mathematical physics and statistical mechanics,
 Langevin equations \cite{Lag08}, and stochastic acceleration  \cite{KP79,KP80} among others, were studied in
 \cite{CF05,Fre04} for the Smoluchowski-Kramers approximation,
 \cite{Freidlin,Cheng}
 for the LDPs and MDPs
 (moderate deviations principles) of Langevin equations in the absence of the random environment, and \cite{NY-JMP} for LDPs of Langevin equations under the random environment given by a Markov switching process taking values in a finite set with a fast jump
 rate.
 However, the studies of the subject involving random environment is still scarce.
 Moreover,  almost all of the existing works requires the random fields be independent of the system and/or have specific formulation.

Because
 much attention has been devoted to the study of
 large deviations principles
 (LDPs) for families of stochastic processes given by first-order stochastic differential equations (SDEs),
  LDPs for the first-order SDEs have been relatively well understood.
 Consider the following SDE
 $$
 dY^\eps(t)=b(t,Y^\eps(t),\xi_{t/\eps})dt
 +\sqrt\eps\sigma(t,Y^\eps(t),\xi_{t/\eps})dw(t),
 $$
 where $w(t)$ is a standard Brownian motion and $\xi_{t/\eps}$ is a random process that may or may not depend on $w(t)$ and $Y^\eps(t)$.
 When $\xi_{t/\eps}$ is deterministic almost surely, the study of large deviations is an
 extension of the Freidlin-Wentzell theory \cite{Fre84}.
 When $\xi_{t/\eps}$ is a random process independent of $B(t)$ and $Y^\eps(t)$,
 the LDPs of $\{Y^\eps\}_{\eps>0}$ has been addressed in \cite{Lip96} for $\xi_{t/\eps}$ being a fast diffusion process having coefficients independent of $Y^\eps(t)$ and driven by another Brownian motion independent of $w(t)$, in \cite{Gui03} for $\xi_{t/\eps}$ being an exponentially erogidic process taking values in a general measurable space, and in \cite{Yinhe} for  $\xi_{t/\eps}$ being a Markovian switching process taking values in a finite set.
 From another angle,
 much effort has been devoted to the study of the coupled system (i.e., $\xi_{t/\eps}$ depending on $Y^\eps(t)$ and $w(t)$). Perhaps one of the natural expressions is to assume $\xi_{t/\eps}$ to be a solution of a fast-varying stochastic differential equation in the setting of fast-slow SDEs. Such cases have been studied in \cite{Ver99,Ver00} for some coupled systems in which some coefficients do not depend on both slow and fast processes,
 and in \cite{Puh16}
 for fully-coupled systems with all coefficients depending on both slow and fast processes and with the driving noises being  correlated.
 Moreover,  fully-coupled systems in which $\xi_{t/\eps}$ being a jump process taking values in a finite set was  considered in \cite{BDG18}.
 However, in contrast to the works on the first-order equations, the study on the second-order equations is still scarce.

\para{Our contributions.}
In this work, we provide a
new approach
compared with \cite{Freidlin,Cheng,NY-JMP} as well as the works for first-order equations, and generalize the results in \cite{Freidlin,CF05,Fre04,NY-JMP} in statistical physics applications.
We study  Langevin equations in a random environment that is not assumed to have a specific form and allowed to interact with the system.
The formulations on state space and the process $\xi_t$ are main
challenges and require careful handling
and new approaches.
Without specific structure of  $\xi_t$, we could not have a
representation formula as in \cite{Bou98,BDM11} for the solution processes. Thus  the weak convergence approach of \cite{Freidlin} and the results in \cite{NY-JMP} are no longer applicable  because we do not assume any specific structure of $\xi_t$ and do not assume the noise moves
much faster.
By establishing the LDP for a Langevin dynamics in random fields,
we provide some insight into the statistical inference for the motions of a net of small particles, which
 is shown to be equivalent to  homogeneous environment obtained by averaging.
This fact plays an important role in practice because, typically, the heterogeneity is often much difficult to analyze and simulate than the homogeneity.
We also generalize the principle of least action
for the  environment with the presence of the heterogeneity.
From a technical point of view,
it is the first work consider the large deviations of a second-order stochastic differential equations in random environment, without specific formulation for the random environment.

\para{Our method and approach.}
Our techniques and method rely on the relation between the solutions of the second-order equations and the associate
first-order equations.
Our approach stems from
the intuition of the Smoluchowski-Kramers approximation.
Our proof of the main results is based on the property that if a family of processes enjoys the exponential tightness and a local LDP with an appropriate
rate function, then it satisfies the LDP with the same rate function.
One of the difficulties stems
from handling the diffusion part of the solution of \eqref{eq:F-setup} with a term $\frac 1{\eps^2}\int_0^t H_\eps(s)ds$, where
$$H_\eps(t):=\sqrt \eps e^{-A_\eps(t)}\int_0^t e^{A_\eps(s)}
\sigma_\eps(s,X^\eps_s)dw(s),\text{ and } A_\eps(t):=\dfrac 1{\eps^2}\int_0^t\lambda_\eps(r,X^\eps_r)dr.$$
The large factor $\frac 1{\eps^2}$ in $\frac 1{\eps^2}\int_0^t H_\eps(s)ds$ requires detailed estimates for $H_\eps(t)$. But we cannot move
  $e^{-A_\eps(t)}$
  inside the stochastic integral in It\^o's sense. But, we do need this random variable ($e^{-A_\eps(t)}$) to balance the large factor $e^{A_\eps(s)}$ inside the stochastic integral.

To prove the exponential tightness, we use an extended Puhalskii's criteria
\cite[Theorem 3.1]{LP92} and
\cite[Remark 4.2]{FeKu}.
The challenge in this part is to estimate $H_\eps(t)$ with high probability, which cannot be handled in It\^o's sense or by  martingale estimates.
 By using regularity of the solution to
 interpret $H_\eps(t)$ in pathwise sense
and its suitable decomposition,
 we are able to obtain desired properties needed for the exponential tightness.
Under the assumption on the local LDP of the family of solutions of associated to the first-order equations (obtained by taking the intuition of the Smoluchowski-Kramers approximation), the family of processes $\{X^\eps\}_{\eps>0}$ satisfies a local LDP.
Here, we need to connect
the solutions of the second-order and the associate first-order equations.
Although we do
not expect they are exponentially equivalent,
we expect that they are exponentially equivalent in a ``local sense".
The challenge here is that we need a term, which leverages the decay (as $\eps\to0$) of the distance between solutions of the second-order and the first-order equations.
By looking at the behavior around a fix function $\varphi$ from  auxiliary frozen systems, we are able to replace $H_\eps$ by a stochastic process, in which we can move the random variable from outside to inside the stochastic integral.
It is also noted that even when these types of stochastic integrals can be understood in It\^o's sense, they are no longer martingales. However, by using
techniques
borrowed from handling stochastic convolution, used in stochastic partial differential equations, we can obtain the desired estimates.

\para{Outline of the paper.} The rest of paper is arranged as follows.
Section \ref{sec:for} formulates the problem and states our main results.
Section \ref{sec:app} is devoted to some specific $\xi$'s, applications, and discussions.
The proofs of main results are given in Section \ref{sec:prof}.
To avoid the interruption, proofs of some technical results needed in the proof of main results are postponed to an appendix.

\section{Formulation and Main Results}\label{sec:for}

We use $\abs{\cdot}$ to denote the Euclid norm for vectors or matrices,
$\langle{\cdot,\cdot}\rangle$ the inner product, and 
$\CC([0,1],\R^d)$ the space of continuous functions on $[0,1]$ endowed with the sup-norm $\|\cdot\|$. Denote by
 $\nabla_t$ and $\nabla_x$  the partial
derivatives with respect to the variables $t$ and $x$, respectively.
We work with $(\Omega,\F,\{\F_t\}_{t\geq 0},\PP)$, a complete filtered probability space with the filtration satisfying the usual condition.
Let $w(t)$ be an $m$-dimensional standard Brownian motion and $\xi(t)$ be a random process that may or may not depend on $w(t)$ and that
take values in a measurable space $\M$.
We use
the letter $C$ with or without subscripts to represent a generic
positive
constant, whose values
may
change for different usage. The letters $\hat C$ and $\wdt C$ with or without subscripts are
constants to be specified later.
The constants $C$, $\hat C$, and $\wdt C$ are independent of $\eps$.
We begin with the following definition; see e.g., \cite{DZ98}.

\begin{deff} {\rm
	A family of stochastic processes $\{Y^\eps\}_{\eps>0}$ in $\CC([0,1],\R^d)$
is said to enjoy the large deviations principle
(LDP) with a
rate function $I$
if the following conditions are satisfied:
	\begin{itemize}
		\item $I:\CC([0,1],\R^d)\to[0,\infty]$ is inf-compact, that is, the
		level sets $\{I(f)\leq L\}$ are compact in $\CC([0,1],\R^d)$ for any $L>0$.
		\item For any open subset $G$ of $\CC([0,1],\R^d)$,
		$$
		\liminf_{\eps\to 0}\eps\log\PP(Y^\eps\in G)\geq -I(G):= -\inf_{f\in G}I(f).
		$$
		\item For any closed subset $F$ of $\CC([0,1],\R^d)$,
		$$
		\limsup_{\eps\to 0}\eps\log\PP(Y^\eps\in F)\leq -I(F):= -\inf_{f\in F}I(f).
		$$
	\end{itemize} }
\end{deff}
Our main goal of this paper is to establish a LDP for the family of processes $\{X^\eps=(X^\eps_t)_{t\in[0,1]}\}_{\eps>0}$, which are solutions of the second-order stochastic differential equations (SDEs) with random environment given by
\begin{equation}\label{qe}
\begin{cases}
\eps^2\ddot{X}^\eps_t=b(t,X^\eps_t,\xi_{t/\eps})-\lambda_\eps(t,X^\eps_t)\dot{X}^\eps_t
+\sqrt{\eps}\sigma_\eps(t,X^\eps_t)\dot w(t),\\
X^\eps_0=x_0\in\R^d,\quad\dot{X}^\eps_0=x_1\in \R^d.
\end{cases}
\end{equation}
With $X^\eps$ denoting
the solution of \eqref{qe}, the
pair
$(X^\eps,p^\eps)$ is the solution of the following system of  first-order SDEs
\begin{equation}\label{pq}
\begin{cases}
\dot{X}^\eps_t=p^\eps_t,\quad X^\eps_0=x_0\in\R^d,\\
\eps^2\dot{p}^\eps_t=b(t,X^\eps_t,\xi_{t/\eps})-\lambda_\eps(t,X^\eps_t)p^\eps_t
+\sqrt{\eps}\sigma_\eps(t,X^\eps_t)\dot w(t),\quad p^{\eps}_0=x_1\in\R^d.
\end{cases}
\end{equation}
For simplicity, we  assume $x_0,x_1$ to be non-random and fixed.
More general cases can be handled similarly; see Remark \ref{rem-3}.
To proceed,
we  make the following assumptions on the coefficients of \eqref{qe}, almost of them are similar to that used in the literature;
 see e.g., \cite{Freidlin,Cheng,NY-JMP}),
 and the assumption on the local LDPs for the corresponding first-order equations.

\begin{asp}\label{asp-1}
{\rm	Suppose that
	\begin{itemize}
		\item	$b(\cdot,\cdot,\cdot):\R_+\times\R^d\times\M\rightarrow\R^d$ is measurable and that there exists a constant $C>0$ satisfying
		$$\abs{b(t,x,\xi)-b(t,y,\xi)}\leq C\abs{x-y}, \ \text{for all}\ t\geq 0, x\in\R^d, \xi\in \M;$$
		\item for each $\eps>0$, $\sigma_\eps(t,x): \R_+\times\R^d \mapsto \R^{d\times m}$
		is continuously differentiable functions with respect to $t$ and $x$ and satisfies
		$$ \limsup_{\eps\to0}\sup_{(t,x)\in\R_+\times\R^d}\|\sigma_\eps(t,x)\|+\|[\sigma_\eps(t,x)]^{-1}\|+\|\nabla_t\sigma_\eps(t,x)\|<\infty;\; \sup_{(t,x)\in\R_+\times\R^d}\|\nabla_x\sigma_\eps(t,x)\|\leq C\eps^2;
		$$
		\item for each $\eps>0$, the mapping $\lambda_\eps(\cdot,\cdot):\R_+\times \R^d\mapsto \R$ 
		twice continuously differentiable functions
		and satisfies
		$$ \limsup_{\eps\to0}\sup_{(t,x)\in\R_+\times\R^d}|\lambda_\eps(t,x)|+|\nabla_t\lambda_\eps(t,x)|<\infty;\;
		\sup_{(t,x)\in\R_+\times\R^d}|\nabla_x\lambda_\eps(t,x)|+|\nabla_{xx}\lambda_\eps(t,x)|\leq C\eps^2;
		$$
		$$\kappa_0 :=\liminf_{\eps\to0}\inf_{(t,x)\in\R_+\times\R^d}\lambda(t,x)>0.$$
	\end{itemize} }
\end{asp}

\begin{asp}\label{asp-2}
	{\rm [Assumptions on the first-order equation]. Assume
 that the family $\{q^\eps=(q^\eps_t)_{t\in[0,1]}\}_{\eps>0}$ of solutions of the following stochastic differential equation
	\begin{equation}\label{qe-000}
	\begin{cases} \dot{q}^\eps_t=\dfrac{b(t,q^\eps_t,\xi_{t/\eps})}{\lambda_\eps(t,q_t^\eps)}+\sqrt{\eps}\dfrac{ \sigma_\eps(t,q_t^\eps)}{\lambda_\eps(t,q_t^\eps)}\dot w(t),\\
	q^\eps_0=x_0\in\R^d,
	\end{cases}
	\end{equation}
	satisfies the local LDP (see Definition \ref{def-localLDP}) in $\CC([0,1],\R^d)$ with a rate function $\hat I(\cdot)$,
and
$\hat I(\varphi)=\infty$ if $\varphi$ is not absolutely continuous.}
\end{asp}

\begin{rem}{\rm
	Equation \eqref{qe-000} is obtained by using the Smoluchowski-Kramers approximation. Intuitively, when $\eps\to0$, the terms being higher-order of $\eps$ in the second-order equation \eqref{qe} converge to $0$ much faster and then stay around $0$ for a long time
(compared with other terms).
	Therefore, roughly, as $\eps\to 0$, equation \eqref{qe}
is approximated by
	$$
	0=b(t,X^\eps_t,\xi_{t/\eps})-\lambda_\eps(t,X^\eps_t)\dot X^\eps_t+\sqrt\eps\sigma_\eps(t,X^\eps_t)\dot w(t).
	$$
	As a result, we obtain \eqref{qe-000}.
}\end{rem}

\begin{rem}{\rm
	As we mentioned briefly in the introduction,
	it is well-known that
Assumption \ref{asp-2} on the local LDPs of the first-order SDEs \eqref{qe-000} is
not restrictive.
The LDPs for the first-order SDEs in random environment are well-established with explicit rate function in the literature
under
 different formulations of different random processes $\xi_t$;
see some examples of our results in Section \ref{sec-exp}.
}\end{rem}

Our main result is
the LDP for the family of solutions of \eqref{qe}.
 This is given in the following theorem.

\begin{thm}\label{thm-main1}
	Under Assumptions {\rm \ref{asp-1}} and {\rm \ref{asp-2}}, the family  of solutions $\{X^\eps\}_{\eps>0}$ of \eqref{qe} satisfies the LDP in $\CC([0,1],\R^d)$ with rate function $\hat I(\cdot)$ in
Assumption \ref{asp-2}.
\end{thm}

\noindent{\bf Specification:} Not only
	Theorem \ref{thm-main1} gives us a limit as $\eps\to 0$, but also tail probability estimate of the convergence. Let us assume that $\hat I(\varphi)=0$ has unique solution $\varphi^*$. [If we assume a specific form of $\xi_{t/\eps}$, we can obtain explicit formula for $\hat I$. In such a case,  it can be verified that $\hat I$ verifies the condition assumed in this remark; see Section \ref{sec-exp}.]
	Let $B(\varphi^*)$ be an arbitrary neighborhood of $\varphi^*$ and $B^c(\varphi^*)$ be its completion in $\CC([0,1],\R^d)$.
	We have $\hat I(B^c(\varphi^*)) > 0$. Otherwise, if $\hat I(B^c(\varphi)) = 0$, there exists
	$\{\varphi_k\}_{k=1}^\infty \subset B^c(\varphi^*)$ such that $\lim_{k\to\infty} \hat I(\varphi_k) = 0$. Due to $\hat I$ is good rate function, there exists
	a convergent subsequence (still denoted by $\varphi_k$) and with limit $\varphi^{**}\in B^c(\varphi^*)$.
	Since $\hat I$ is lower semi-continuous, $0 \leq \hat I(\varphi^{**}) = \hat I(\lim_{k\to\infty}\varphi_k) \leq  \lim_{k\to\infty} \hat I(\varphi_k) = 0$. It leads to
	$\hat I(\varphi^{**}) = 0$, which is a contradiction.
	That means $\hat I(B^c(\varphi^*)) > 0$. Thus, from the LDP of $\{X^\eps_t\}_{\eps>0}$, the probability $\PP(X^\eps_t\in B^c(\varphi^*))\approx \exp\{-\frac {\hat I(B^c(\varphi^*))}{\eps}\}$ tends to 0
	exponentially fast.

\begin{rem}\label{rem-3}
{\rm	It will be seen in the proof of Theorem \ref{thm-main1} that the LDP of $\{X^\eps\}_{\eps>0}$ still holds if the initial value $x_1$ depending on $\eps$ (i.e., $x_1=x^\eps_1$) satisfies $\eps x^\eps_1$ being bounded as $\eps\to 0$. For example,
we may replace the initial condition $x_1$  by
$x^\eps_1= {x_1}/{\eps}$, which occurs in some applications to physics after scaling the time.
	In a more general setting, we can also allow both $x_0,x_1$ to be random and depending on $\eps$ as well. To be more specific, we can replace the initial values $x_0,x_1$ by $x_0^\eps,x_1^\eps$ and assume that $\limsup_{\eps\to 0}\eps |x_1^\eps|<\infty$ a.s. and $\{x_0^\eps\}_{\eps>0}$ obeys the LDP (in $\R^d$) with the rate function $I_0$ so that the rate function for \eqref{qe-000} becomes $I_0+\hat I$ in the sense that $(I_0+\hat I)(\varphi)=I_0(\varphi_0)+\hat I(\varphi)$ (see \cite[Section 9]{Puh16}). Then Theorem \ref{thm-main1} still holds with the rate function $I_0+\hat I$.}
\end{rem}

\begin{rem} {\rm
The results presented in Theorems \ref{thm-main1}, as well as others in Section \ref{sec-exp}
can be extended to the space $\CC([0,T],\R^d)$ of
	continuous functions on $[0,T]$ endowed with the sup-norm topology for any $T>0$.
	As a consequence, these LDPs still hold in $\CC([0,\infty),\R^d)$, the space of continuous function on $[0,\infty)$ endowed with the local supremum topology
	defined by the metric
	$$
	\sum_{n=1}^\infty\frac 1{2^n}\left(1\wedge\sup_{t\leq n} |\varphi_t -\psi_t|\right),\quad\forall \varphi,\psi\in\CC([0,\infty),\R^d).
	$$
	This fact follows from the Dawson-G\"artner theorem; see \cite[Theorem 4.6.1]{DZ98}, which states that it is sufficient to check the LDPs in $\CC([0, T],\R^d)$ for any $T$ in the uniform metric.}
\end{rem}

\section{Specifications, Examples,
	and Discussions}\label{sec:app}
In this section, we first provide
several
specifications of the process $\xi_{t/\eps}$. Then we consider some examples in statistical mechanics.

\subsection{Special Cases of $\xi_{t/\eps}$:
	Diffusion, Jump, and Switching Processes}\label{sec-exp}

\noindent
{\bf Diffusion processes.}
We assume the noise process $\xi_t$ is given by a diffusion and then $\xi_{t/\eps}$ is a fast diffusion in $\M=\R^l$. Thus, we consider the second-order system involving fast and slow  processes
\begin{equation}\label{qe-exp-1}
\begin{cases}
\eps^2\ddot{X}^\eps_t=b(t,X^\eps_t,Y^\eps_t)-\lambda_\eps(t,X^\eps_t)\dot{X}^\eps_t
+\sqrt{\eps}\sigma_\eps(t,X^\eps_t)\dot w(t),\\
\disp \dot{Y}^\eps_t=\frac 1{\eps}F(t,X^\eps_t,Y^\eps_t)+\frac 1{\sqrt\eps}G(t,X^\eps_t,Y^\eps_t)\dot {\wdt w}(t),\\
X^\eps_0=x_0\in\R^d,\quad\dot{X}^\eps_0=x_1\in \R^d,\quad Y_0^\eps=y_0\in\R^l,
\end{cases}
\end{equation}
where $Y^\eps_t\in\R^l$,
$F(t,x,y):\R_+\times\R^d\times\R^l\to\R^l$, $G(t,x,y):\R_+\times\R^d\times\R^l\to\R^{l\times n}$ are measurable functions, and
$\wdt w(t)$ is an $n$-dimensional
standard
Brownian motions.
Moreover, we allow $\wdt w(t)$ to be correlated with $w(t)$ and denote its correlation matrix by $\Sigma$ (i.e., $\Sigma$ is a $m\times n$ matrix and its $(i,j)$-th entry is the correlation of the $i$-th component of $w(t)$ and the $j$-th component of $\wdt w(t)$).

Together with  Assumption \ref{asp-1} for the coefficient $b,\lambda,\sigma$, we make the following assumptions for $F,G$ (see \cite{Puh16}), which is much milder  than that of $b,\lambda,\sigma$.
The smoothness conditions on $b,\lambda,\sigma$ have two main purposes. First, the conditions enable us
to treat the stochastic integral, which cannot be estimated as in the case of the first-order SDEs. Second, the conditions are needed to establish a (local sense)
exponential equivalence
of the solutions of the second-order SDEs and its associated first-order SDEs.

\begin{asp}\label{asp-exp-1} {\rm
		The functions $G(t,x,y)$ (as well as $G(t,x,y)[G(t,x,y)]^\top)$ are bounded locally in $(t, x)$ and globally in $y$ and are continuous in
		$(x, y)$.
		The function $F(t,x,y)$ is measurable and locally bounded in $(t,x,y)$ and is
		Lipschitz continuous in $y$ and continuous locally uniformly in $(t, x)$. The functions $F(t,x,y)$ and
		$G(t,x,y)[G(t,x,y)]^\top$ are continuous in $x$ locally uniformly in $t$ and uniformly in $y$. $G(t,x,y)[G(t,x,y)]^\top$
		is
		of class $\CC^1$ in $y$, with the first partial derivatives being bounded and Lipschitz continuous in $y$ locally uniformly in $(t, x)$, and $\text{\rm div}_y G(t,x,y)[G(t,x,y)]^\top$ is continuous in $(x,y)$.
		Moreover, for any $t, N>0$,
		$$
		\lim_{|y|\to\infty}\sup_{s\in[0,t]}\sup_{x\in\R^d:|x|<N}\frac{[F(s,x,y)]^\top y}{|y|^2}<0.
		$$ }
\end{asp}
Let $\lambda_0(t,x),\sigma_0(t,x)$ is the limit of $\lambda_\eps(t,x),\sigma_\eps(t,x)$ in the sense of that
$$
\limsup_{\eps\to0}\sup_{(t,x)\in\R_+\times\R^d}(|\lambda_0(t,x)-\lambda_\eps(t,x)|+|\sigma_0(t,x)-\sigma_\eps(t,x)|)=0.
$$

\begin{asp}\label{asp-exp-12}
	{\rm	The matrix $G(t,x,y)[G(t,x,y)]^\top$ is positive definite uniformly in $y$ and
		locally uniformly in $(t,x)$. Either $\sigma_0(t,x) = 0$ for all $(t,x)\in\R_+\times\R^d$ or the matrix
		$$
		G(t,x,y)[G(t,x,y)]^\top - \sigma_0(t,x)\Sigma[G(t,x,y)]^\top\Big(G(t,x,y)[G(t,x,y)]^\top\Big)^{-1}G(t,x,y)\Sigma^\top[\sigma_0(t,x)]^\top
		$$
		is positive definite uniformly in $y$ and locally uniformly in $(t, x)$.}
\end{asp}

Applying Theorem \ref{thm-main1} and \cite[Corollary 2.1]{Puh16}, we have the following result.

\begin{thm}
	Assume assumptions {\rm\ref{asp-1}}, {\rm\ref{asp-exp-1}}, and {\rm\ref{asp-exp-12}} hold.
	The family of processes $\{X^\eps\}_{\eps>0}$ satisfies the LDP in $\CC([0,1],\R^d)$ with the rate function $I$ given as follows. If $\varphi$ is absolutely continuous and $\varphi_0=x_0$, then
	$$
	\begin{aligned}
	I(\varphi)=\int_0^1 \sup_{\beta\in\R^d} \Bigg(&\beta^\top\dot\varphi_s-\sup_{m\in\mathcal P(\R^l)}\bigg(\beta^\top \int_{\R^l}\frac{b(s,\varphi_s,y)m(y)}{\lambda_0(s,y)}dy+\frac 12 \beta^\top \Big(\int_{\R^l}\frac{\sigma_0(s,y)[\sigma_0(s,y)]^\top}{\lambda_0(s,y)}m(y)dy\Big)\beta\\
	&-\sup_{h\in\CC^1_0(\R^l)}\int_{\R^l}\Big([\nabla h(y)]^\top\Big(\frac 12\text{\rm div}_y(G(s,\varphi_s,y)[G(s,\varphi_s,y)]^\top m(y))\\
	&\hspace{4.3cm}-F(s,\varphi_s,y)m(y)- G(s,\varphi_s,y)\Sigma^\top[\sigma_0(s,y)]^\top\beta m(y)\Big)\\
	&\hspace{3cm}-\frac12 [\nabla h(y)]^\top G(s,\varphi_s,y)[G(s,\varphi_s,y)]^\top [\nabla h(y)]\Big)dy\bigg)\Bigg)ds.
	\end{aligned}
	$$
	Otherwise, $I(\varphi)=\infty.$
	In the above, $\mathcal P(\R^l)$ is the space of probability density functions $m(s)$ in $\R^l$  such that $m \in \mathbb W^{1,1}_{\rm loc}(\R^l)$ and
	$\sqrt m \in \mathbb W^{1,2}(\R^l)$ with $\mathbb W^{1,2}(\R^l)$, $\mathbb W^{1,1}_{\rm loc}(\R^l)$ being Sobolev $($and local Sobolev$)$ spaces with appropriate $($indicated$)$ exponents in $\R^l$,
	$\CC^1_0(\R^l)$ is the space of continuously differentiable functions with compact support in $\R^l$.
\end{thm}

\

\para
{Jump processes.}
Here,
we assume $\xi_t$ is a jump process taking finite values, which depends on the slow process as well.
To be more precise, assume $\M=\{1,\dots,|\M|\}$ is a finite set.
Similar to \cite{BDG18}, the evolution of the jump fast component is constructed through a jump intensity function $c(x,y)=c_y(x) : \R^d \times\M\to [0, \infty)$ and a transition probability function $r(x,y,y')=r_{yy'}(x) : \R^d\times\M \times \M \to [0, 1]$ as follows.

Assume that for all $(x,y)\in\R^d\times\M$, $\sum_{y'\in\M} r_{yy'}(x)=1, r_{yy}(x) = 0.$
Let $\zeta=\sup_{(x,y)\in\R^d\times\M}c_y(x)+1$,
$E_{yy'}(x)=[0,c_y(x)r_{yy'}(x)]$ for all $(x,y,y')\in\R^d\times\M\times\M$, $y\neq y'$, and
$\mathbb T =: \{(y,y') \in\M\times\M : r_{yy'}(x) > 0$  for some $x \in\R^d\}$.
For
$(i,j)\in\mathbb T$, let $\bar N_{ij}$ be a Poisson random measure on $[0; \zeta] \times [0,T] \times \R_+$ with intensity measure $\mu_{\zeta}\otimes \mu_T \otimes \mu_{\infty}$, where $\mu_T$ and $\mu_\infty$ denote the Lebesgue measures on $[0, T]$
and $\R_+$, respectively such that
for $t\in[0,T]$,
$$
\bar N_{ij}(A \times [0, t] \times B) - t\mu_\zeta(A)\mu_\infty(B)
$$
is an $\F_t$-martingale for all $A\in \mathcal B[0,\zeta]$ and $B\in \mathcal B(\R_+)$ with $\mu_\infty(B) < 1$.
Then, we define
$$
N^{\eps^{-1}}
_{ij} (dr\times dt) = \bar N_{ij}(dr \times dt \times [0,\eps^{-1}])
$$
a Poisson random measure on $[0, \zeta]\times [0, T]$ with intensity measure $\eps^{-1}\mu_\zeta\otimes\mu_T$.
The processes $(N^{\eps^{-1}}_{ij})_{(i,j)\in\mathbb T}$ are taken to be mutually independent. We assume that for $0 \leq s \leq t \leq T$,
$$
\{w(t)-w(s); N^{\eps^{-1}}_{ij}(A \times (s; t] \times B): A \in \mathcal B[0,\zeta], B \in\mathcal B(\R_+), (i,j)\in
\mathbb T\}
$$
is independent of $\F_s$.
Now, we consider the following system
\begin{equation}\label{qe-exp-2}
\begin{cases}
\eps^2\ddot{X}^\eps_t=b(X^\eps_t,Y^\eps_t)-\lambda\dot{X}^\eps_t
+\sqrt{\eps}\sigma\dot w(t),\\
dY^\eps_t=\sum_{(i,j)\in\mathbb T}\int_{r\in[0,\zeta]}(j-i)\1_{\{Y^\eps(t-)=i\}}\1_{E_{ij}(X^\eps_t)}(r)N_{ij}^{\eps^{-1}}(dr\times dt),\\
X^\eps_0=x_0\in\R^d,\quad\dot{X}^\eps_0=x_1\in \R^d,\quad Y_0^\eps=y_0\in\M.
\end{cases}
\end{equation}
According to \cite{BDG18}, we make following assumption for the jump process and construct the rate function as follows.

\begin{asp}\label{asp-exp-2}
	{\rm	
		Function $c$ is bounded and
		there exists finite constant $C>0$ such that for all $y,y'\in \M$ and $x_1, x_2\in \R^d$,
		$$
		|c_{y}(x_1) - c_y(x_2)|+|r_{yy'}(x_1) - r_{yy'}(x_2)|
		\leq C|x_1 - x_2|.
		$$
		Moreover,
		$$	 \inf_{x\in\R^d}\min_{y,z\in\M}\sum_{n=1}^{|\M|}r^n_{yz}(x)>0,\quad \inf_{x\in\R^d}\min_{y\in\M}c_y(x)>0,\quad
		\inf_{x\in\R^d}\min_{(y,y')\in\mathbb T}r_{yy'}(x)>0.
		$$
}\end{asp}

For $\psi=(\psi(j))_{j\in\M}$, with $\psi_j:[0, \zeta] \to \R_+$ being a measurable map for every $j$, define
$$
\Phi^{\psi}_{ij}(x) =
\begin{cases}
\int_{E_{ij}(x)}\psi_j(z)\mu_\zeta d(z),\text{ if }i\neq j,\\
-\sum_{y:y\neq j}\Phi^{\psi}_{jy}(x),\text{ if } i=j,
\end{cases}
$$
and
$$
\mathcal R =\{v=(v_{ij})_{(i,j)\in\mathbb T}, v_{ij}: [0,1]\times[0,\zeta]\to\R_+\text{ is measurable for all } (i,j)\in\mathbb T\}.
$$
For $\varphi\in C([0,1],\R^d)$, let $\mathcal V(\varphi)$ be the collection of all
$$
\Big(u=(u_i),v=(v_{ih}),\pi=(\pi_i)\Big)\in\mathbb M([0,1]:\R^m)^{|\M|}\times\mathcal R\times \mathbb M([0,1]:\mathcal P(\M)),
$$
[where $\mathbb M([0, 1] : \mathcal P(\M))$, $\mathbb M([0,T] : \R^d)$ denote the space of measurable maps from
$[0, 1]$ to $\mathcal P(\M)$ and from $[0,1]$ to $\R^d$, respectively, with $\mathcal P(\M)$ being the space of probability measures on $\M$ equipped with the topology of weak convergence], such that $\int_0^1\|u_i(s)\|^2\pi_i(s)ds<\infty$ for each $i\in\M$, and
$$
\varphi_t=x_0+\sum_{j\in\M}\int_0^t \frac{b(\varphi_s,j)}{\lambda}\pi_j(s)ds+\sum_{j\in\M}\frac{\sigma u_j(s)\pi_j(s)}{\lambda}ds,
$$
and
$$
\sum_{j\in\M}\pi_j(s)\Phi^{v_{j\cdot}(s,\cdot)}_{ji}(\varphi_s)=0, a.e.\; s\in[0,1], \forall i\in\M.
$$
Combining Theorem \ref{thm-main1} and \cite{BDG18} yields the following result.

\begin{thm}
	Let assumptions {\rm\ref{asp-1}} and  {\rm\ref{asp-exp-2}} hold.
	Then the family of processes $\{X^\eps\}_{\eps>0}$ satisfies the LDP in $\CC([0,1],\R^d)$ with the rate function $I$ given by
	\begin{equation}\label{eq-exp2-rate}
	I(\varphi)=\displaystyle\inf_{(u,v,\pi)\in\mathcal V(\varphi)}\Bigg\{\sum_{i\in\M}\frac 12\int_0^1\|u_i(s)\|^2\pi_i(s)ds+\sum_{(i,j)\in\mathbb T}\int_{[0,\zeta]\times[0,1]}\ell(v_{ij}(s,z))\pi_i(s)\mu_\zeta(dz)ds\Bigg\},
	\end{equation}
	where $\ell(x) = x \ln x - x + 1$.
\end{thm}

Note that
another representation for the rate function $I$ can be found in \cite[(2.18)]{BDG18}.

\

\para{Markov chains.}
We consider $\xi_{t/\eps}=\alpha^\eps_t$ to be a Markov
switching process independent of the Brownian motion $w(t)$ taking value in a finite state space $\M$ such that $\alpha^\eps_t$ has generator $Q(t)/\eps$ with
$Q(t)\in \R^{|\M|\times|\M|}$ being a generator of a continuous-time, irreducible
Markov chain.
Consider the following system
\begin{equation}\label{eq-exp-3}
\begin{cases}
\eps^2\ddot{X}^\eps_t=b(t,X^\eps_t,\alpha^\eps_t)-\lambda_\eps(t,X^\eps_t)\dot{X}^\eps_t
+\sqrt{\eps}\sigma_\eps(t,X^\eps_t)\dot w(t),\\
X^\eps_0=x_0\in\R^d,\quad\dot{X}^\eps_0=x_1\in \R^d.
\end{cases}
\end{equation}
As stated in Theorem \ref{thm-main1}, the family of solutions of \eqref{eq-exp-3} satisfies the LDP as well.
The rate function can be established as follows (see e.g., \cite[Theorem 4.3] {Yinhe})\footnote{Another possible approach is to specialize the rate function from \eqref{eq-exp2-rate} after formulating the Markovian switching in sense of jump process}:
$$
I(\varphi)=
\begin{cases}
\displaystyle\int_0^1L(\varphi_s,\dot\varphi_s,s)ds\text{ if }\varphi \text{ is absolutely continuous},\\
\infty, \text{ otherwise}.
\end{cases}
$$
In the above, $L(x,\gamma,s)$ is the Fenchel-Legendre transform of the $H$-functional, i.e.,
$$
L(x,\gamma,s):=\sup_{\beta\in\R^d}[\langle \gamma,\beta\rangle -H(x,\beta,s)]
$$
and
$H(x,\beta,s)$ is the function such that (see e.g., \cite[Lemma 4.1]{Yinhe} or \cite[Theorem 1]{Ver00} for the proof of its existence and properties)
$$
\lim_{\eps\to 0}
\eps \log \E_i \exp \Bigg\{\frac 1\eps \int_0^T \bigg(\frac{[\beta_s]^\top b(s,\varphi_s,\beta_s)}{\lambda_\eps(s,\varphi_s)}+\frac{|[\sigma_\eps(s,\varphi_s)]^\top\beta_s|^2}{2\lambda^2_\eps(s,\varphi_s)} \bigg) ds\Bigg\}= \int_0^T H(\varphi_s, \beta_s, s)ds
$$
for any step functions $\varphi_s$ and $\beta_s$ in $\R^d$ and $\E_i$ indicates the expectation with respect to the initial value $\alpha^\eps(0)=i$.

\subsection{Statistical Mechanics Examples}

In this section, we consider some
examples
in
statistical mechanics of small particles in a random environment, which generalizes that
of \cite{Freidlin,Cheng,NY-JMP}.
Let us assume that $\xi_{t}$ is a ergodic process, which may depend on $X_t^\eps$, and assume for each fixed $t$ and fixed state $X^\eps_t=x$, $\xi_{t/\eps}$, as $\eps\to0$, has invariant measure denoted by $\pi_{t,x}$.

\para{Smoluchowski-Kramers approximation.} Consider an overdampped approximation
\begin{equation}\label{oLE-4}
\dot{q}^\eps_t=\frac{b(t,q^\eps_t,\xi_{t/\eps})}{\lambda_\eps(t,q^\eps_t)}
+\sqrt{\eps}\frac{\sigma_\eps(t,q^\eps_t)}{\lambda_\eps(t,q^\eps_t)}\dot w(t).
\end{equation}
In homogeneous environment, it is well-known that the Langevin equation can be simplified to the overdampped approximation,
which is also commonly referred to as a
Smoluchowski-Kramers approximation \cite{KN91}.
Our result in Proposition \ref{close}, one of main steps in the proof of our main result, given in Section \ref{sec:local} is a generalization of this classical result to the case of presence of another interacting random process.

\para{Principle of Least Action.} As seen in Section \ref{sec-exp}, under certain conditions,
the family of solutions $\{q^\eps\}_{\eps>0}$ of overdampped approximations \eqref{oLE-4} satisfies the LDP with the rate function denoted by $\hat I(\cdot)$, which can be given explicitly as in Section \ref{sec-exp} depending the formulation of $\xi_{t/\eps}$.
Applying our results, the solution $\{X^\eps\}_{\eps>0}$ of \eqref{qe}
satisfies the LDP
with the same rate function $\hat I$.
Theorem \ref{thm-main1}
shows that
\begin{equation}
\begin{aligned}
-\inf_{\varphi\in B^\circ}\hat I(\varphi)\leq \liminf_{\eps\to 0}\eps\log \PP\{X^\eps\in B\}
\leq \limsup_{\eps\to 0}\eps\log \PP\{X^\eps\in B\}
\leq -\inf_{\varphi\in\bar B}\hat I(\varphi),
\end{aligned}
\end{equation}
where $B^\circ$ and $\bar B$ denote the interior and closure of $B$ in $\CC([0,1],\R^d)$.

If we let $P_\eps[X]$ be the probability density functional or
law over  different trajectories $\{X^\eps\}_{\eps>0}$
in a given time interval $[0, 1]$. Then a LDP for the random paths $\{X^\eps\}_{\eps>0}$  indicates that
$$
P_\eps[\varphi]\sim e^{-\hat I(\varphi)/\eps},\quad\eps\to 0.
$$
We denote by $\varphi^*$, the solution to
$$
\dot{\varphi}^*_t=\frac{\bar b(t,\varphi_t^*)}{\lambda_0(t,\varphi_t^*)},
$$
where
\begin{equation}\label{s3-barb}
\bar b(t,x):=\int_{\M} b(t,x,z)\pi_{t,x}(dz),
\end{equation}
and $\lambda_0$ is the limit of functions $\{\lambda_\eps\}_{\eps\to0}$.
It is easy to check that $I(\varphi^*)=0$; see some explicit formulas of $\hat I$ in Section \ref{sec-exp}.
Hence,
the random path of particles in time-inhomogeneous environment contributes around of the path of $\varphi^*$ as $\eps\to0$ with exponential tail, i.e., the probability of
the
path of $X^\eps$  far from that of $\varphi^*$ is exponentially small.
In addition, because of the formula of $\bar b(\cdot,\cdot)$,  one sees that the ``equilibrium" $\varphi^*$ is obtained by considering the particle in a new environment, which averages the time-inhomogeneous environment. So, in the random environment changing in time and state, we can know clearly the statistical physics of small particles if we know the invariant measure $\pi_{t,x}$ of $\xi_{t/\eps}$ describing how the external force fluctuates with each fixed time $t$ and state $x$.

We can view $\hat I$ as the action of the system, i.e.,
$$
\hat I[\varphi]=\int L(\varphi_s,\dot {\varphi}_s)ds.
$$
In above, $L(\cdot,\cdot)$ is the Legendre transform of the Hamiltonian $H(\cdot,\cdot)$;
the details
of which can found in \cite{Gre06} and references therein.
The classical principle of least action states that the actual path taken by $\varphi^*$ is an extremum of $\hat I$.
Under this observation, our result generalizes the principle of least action as follows.
The LDP indicates that
\begin{equation*}
\PP\{X^\eps\in B\}\sim e^{-\frac 1{\eps}\inf_{\varphi\in B}\hat I(\varphi)},
\end{equation*}
The probability is determined
by the path that minimizes the rate function. This corresponds
to
minimizing the action in order to find the path
that is taken by the system, where the integral $\int_0^T L(\varphi,\dot{\varphi},s)ds$ is considered as the Action.
Our results show that
the principle of least action still holds in random environment modeled for small particles.
Moreover, the statistical inference for the system in inhomogeneous models can be  obtained by averaging the time-inhomogeneous factors (in the sense of equation \eqref{s3-barb}).
This fact plays an important role in practice because, typically, the heterogeneity is often much difficult to analyze and simulate than homogeneity.

\section{Proof of Main Results}\label{sec:prof}
In this section, we give the proof of
Theorem \ref{thm-main1}.
We begin with some basic definitions and preliminaries of large deviations theory; for further details, we refer the reader to \cite{DS89,DZ98,LP92}.

\begin{deff} {\rm
		A family of stochastic processes $\{Y^\eps\}_{\eps>0}$ is said to be exponentially tight in the space $\CC([0,1],\R^d)$, if there exists an increasing sequence of compact subsets $\{K_L\}_{L\geq 1}$ of $\CC([0,1],\R^d)$ such that
		$$
\lim_{L\to\infty}\limsup_{\eps\to0}\eps\log\PP\left(Y^\eps\notin K_L\right)=-\infty.
		$$ }
\end{deff}

\begin{deff}\label{def-localLDP}
	 {\rm
		A family of stochastic processes $\{Y^\eps\}_{\eps>0}$ is said to  satisfy the local LDP in $\CC([0,1],\R^d)$ with rate function $\hat I$, if for any $\varphi\in\CC([0,1],\R^d)$,
		$$
		\begin{aligned}
		\lim_{\delta\to 0}&\limsup_{\eps\to 0}\eps\log\PP\left(Y^\eps\in B(\varphi,\delta)\right)\\
		&=\lim_{\delta\to 0}\liminf_{\eps\to 0}\eps\log\PP\left(Y^\eps\in B(\varphi,\delta)\right)\\
		&=-\hat I(\varphi),
		\end{aligned}
		$$
		where $B(\varphi,\delta)$ is the ball  centered at $\varphi$ with radius $\delta$ in $\CC([0,1],\R^d)$.}
\end{deff}

The following is a well-known result in large deviations theory; see e.g.,
\cite{DS89,DZ98,LP92}.

\begin{prop}\label{prop-local-LDP}
	The exponential tightness and the local LDP for a family $\{Y^\eps\}_{\eps>0}$ in $\CC([0,1],\R^d)$ with local rate function $\hat I$ imply the full LDP in $\CC([0, 1],\R^d)$ for this family with rate function $\hat I$.
\end{prop}

\subsection{A Road Map}
To help the reading,
we provide a road map of our approach.
To establish the family of processes $\{X^\eps\}_{\eps>0}$ satisfying the LDP, we prove that it enjoys the exponential tightness and the local LDP thanks to Proposition \ref{prop-local-LDP}.

\noindent {\bf Representation of the solution.} First, we carefully examine the  solutions.
Since $X^\eps_t$ is the solution of a second-order differential equation, we need to solve the equation for its derivative first by using the variation of parameter formula.
We rewrite  equation \eqref{qe} as \eqref{pq} and solve \eqref{pq} to find $p^\eps_t$.
Then, we can the formula for $X^\eps_t$ as
$$
X^\eps_t=x_0+x_1\int_0^t e^{-A_\eps(s)}ds+\dfrac 1{\eps^2}\int_0^t\int_0^s e^{-A_\eps(s,r)}b(r,X^\eps_r,\xi_{r/\eps})dr ds+\dfrac 1{\eps^2}\int_0^t H_\eps(s)ds,
$$
where for any $0\leq s\leq t\leq 1, \eps>0$,
$$A_\eps(t,s):=\dfrac 1{\eps^2}\int_s^t\lambda_\eps(r,X^\eps_r)dr,\;\;\;A_\eps(t)=A_\eps(t,0),$$
$$H_\eps(t):=\sqrt \eps e^{-A_\eps(t)}\int_0^te^{A_\eps(s)}
\sigma_\eps(s,X^\eps_s)dw(s).$$
The large factor $\frac 1{\eps^2}$ in $\frac 1{\eps^2}\int_0^t H_\eps(s)ds$ is a main challenge for us to obtain the desired estimates directly. Therefore, we use an integration by parts formula
to overcome the difficulty;
 see Section \ref{sec:rep}.

\noindent {\bf
Exponential tightness.} It suffices to prove $\{X^\eps\}_{\eps>0}$ satisfying the (extended) Puhalskii's criteria (see
\cite[Theorem 3.1]{LP92} and
\cite[Remark 4.2]{FeKu}), which are
\begin{equation}\label{p-cre-1}
\lim_{L\to\infty}\limsup_{\eps\to 0}\eps\log \PP\Big(\norm{X^\eps}>L\Big)=-\infty,
\end{equation}
\begin{equation}\label{p-cre-2}
\lim_{\delta\to 0}\limsup_{\eps\to0}\sup_{s\in [0,1]}\eps\log\PP\Big(\sup_{s\leq t\leq s+\delta}\left|X^\eps_t-X^\eps_s\right|>\ell\Big)=-\infty,\quad\forall \ell>0.
\end{equation}
To prove \eqref{p-cre-1}, one has to  carefully  estimate the term $H_\eps(t)$
and prove that
\begin{equation}\label{eq-p-H}
\lim_{L\to\infty}\limsup_{\eps\to 0}\eps\log \PP\Big(\norm{H_\eps}>L\Big)=-\infty.
\end{equation}
The difficulty is that we cannot move the non-adapted variable $ e^{-A_\eps(t)}$ into the stochastic integral in It\^o's sense and use estimates for martingales.
On the other hand, we need the term ($e^{-A_\eps(t)}$) to balance the large term $e^{A_\eps(s)}$ in the stochastic integral.
In this case, we need to use the regularity of the coefficient to estimate the stochastic integral in pathwise sense (see e.g., \cite{Freidlin}).
We bound the stochastic process $H_\eps(t)$ by a function of $\sqrt\eps\norm{w}$ and use
the LDP for Brownian motions;
see the details in Section \ref{sec:expbound}.

For the ``exponential equi-continuity condition" \eqref{p-cre-2},
the subtlety
is due to the ``non-adaptedness" mentioned above and the ``non-integral form" of $H_\eps(t)$, which require more delicate analysis to estimate its (uniform) changes in small time.
We need to decompose $H_\eps(t)$ as
$$
H_\eps(t)=\sqrt\eps e^{-M_\eps(t)}h_\eps(t),
$$
where  $M_\eps(t):=M_\eps(t,0)$ with $M_\eps(t,s):=A_\eps(t,s)-\bar A_\eps(t,s)$, and
$$\bar A_\eps(t,s):=
\frac 1{\eps^2}\int_s^t \lambda_\eps(r,0)dr,\quad \bar A_\eps(t)=\bar A_\eps(t,0),\quad
h_\eps(t):=e^{-\bar A_\eps(t)}\int_0^t e^{A_\eps(r)}\sigma_\eps(r,X^\eps_r)dw(r).
$$
 This decomposition separates $H_\eps(t)$ into two parts. One of them is $e^{-M_\eps(t)}$, whose changes in time can be estimated by using the regularity of $\lambda$ and the boundedness established in \eqref{p-cre-1}. The
 other is $h_\eps(t)$, in which we  can move $e^{-\bar A_\eps(t)}$ into the stochastic integral in It\^o's sense. Although the integrand now is adapted, such integral is a
  stochastic convolution but
  it is not a martingale. Therefore, we cannot directly  estimate the tail probability $\PP(\sup_t |h_\eps(t)|>L)$ as in the martingale cases.
 By keeping $t$ frozen, $h_\eps(t)$ can be viewed as an element of a sequence of martingales.
  Then the exponential inequality for martingale helps us to estimate the change in small time intervals of $h_\eps(t)$.
  Therefore, by a technical Lemma \ref{lm-1111}, we are able to obtain the desired estimates.
  The details of this part are in Section \ref{sub:expcont}.

\noindent {\bf Local LDPs.}
Using the intuition of the Smoluchowski-Kramers approximation, we  give some local estimates (similar to ``exponential equivalence property", but in the local sense) for $X^\eps_t$ and $q^\eps_t$, the solutions of the associated first-order equation.
The calculations of this part are relatively complex, whose details are in Section \ref{sec:local}.
Our intuitions and ideas are as follows.

Given an absolutely continuous function $\varphi\in \CC([0,1],\R^d)$ and a neighborhood $B(\varphi,\theta)$ (the ball with center $\varphi$ and radius $\theta$ in $\CC([0,1],\R^d)$), we  prove that there is a neighborhood $B(\varphi,\bar\theta)$ of $\varphi$ such that the difference between the probabilities of $X^\eps\in B(\varphi,\theta)$ and the that of $q^\eps\in B(\varphi,\bar\theta)$ is exponentially small.
As was seen, the distance of $X^\eps$ and $q^\eps$ depends on $H^\eps$.
Since we cannot move the non-adapted random variable inside the stochastic integral in the It\^o sense, estimate \eqref{eq-p-H} cannot be improved. We need a term to leverage the decay (as $\eps\to 0$) of the distance between $X^\eps$ and $q^\eps$ and get an  ``exponentially equivalent" property in the local sense.
By looking at the behavior of the families $\{X^\eps\}$ and $\{q^\eps\}$ around a fixed (absolutely continuous) function $\varphi$, we are able to give  estimates for
$$
\left|\PP\left\{\|q^\eps-\varphi\|<\theta\right\}-
\PP\left\{\|X^\eps-\varphi\|<\bar\theta_1\right\}\right|
$$
 depending on $\|H_\eps^\varphi\|$, where
$$H_\eps^\varphi(t)=\sqrt{\eps} e^{-A_\eps^\varphi(t)}\int_0^t e^{A^\varphi_\eps(s)}\sigma_\eps(r,\varphi_r)dw(r),$$
and
$$A_\eps^\varphi(t,s)=\dfrac 1{\eps^2}\int_s^t\lambda_\eps(r,\varphi_r)
dr;\;\;\;A_\eps^\varphi(t)=A_\eps^\varphi(t,0).$$
In contrast to $H_\eps(t)$, we can write
$$H_\eps^\varphi(t)=\sqrt{\eps}\int_0^t e^{-A_\eps^\varphi(t)+A^\varphi_\eps(s)}\sigma_\eps(r,\varphi_r)dw(r),$$ and understand it in It\^o's sense due to the independence of $\varphi_t$ and $w(t)$.
Again, it is noted that $H_\eps^\varphi(t)$ is not a martingale with respect to $t\in [0,1]$. However, by estimating its change in small time intervals, we can obtain the following (Lemma \ref{lm-Hphi})
$$
\PP\left\{\sup_{t\in[0,1]}|H_\eps^\varphi(t)|>\ell\right\}\leq \bar M_1\exp\left\{-\frac{\bar M_2\ell^2}{\eps^2}\right\},\;\forall \ell>0,
$$
for some constants $\bar M_1,\bar M_2$, independent of $\eps$ and $\ell$.
This fact allows us to  neglect $H_\eps^\varphi$ asymptotically.

The terms left by decomposition process when estimating $\PP(\|X^\eps-\varphi\|<\bar \theta)$ are either controlled by $\eps$, or $\PP(\|H_\eps^\varphi\|>\ell)$, or $\PP(\|q^\eps-\varphi\|<\theta)$.
As a consequence, we can study the behavior of $X^\eps$ around $\varphi$ using that of $q^\eps$.
Then, the local LDP of the family of the solutions of the first-order equations allows us to obtain the local LDP for $\{X^\eps\}_{\eps>0}$;
see details in Section \ref{sec:local}.

\subsection{Representation Formula for Solutions}\label{sec:rep}
Under Assumption \ref{asp-1},
 equation \eqref{pq} admits a unique solution $(X^\eps,p^\eps) \in \CC([0,1],\R^{2d})$. Therefore, equation \eqref{qe} has a unique solution $X^\eps\in \CC([0,1],\R^d).$
From equation \eqref{pq}, by a variation of parameter formula, we  obtain
\begin{equation}\label{formulap}
\begin{aligned}
p^\eps_t=x_1e^{-A_\eps(t)}+\dfrac 1{\eps^2}\int_0^te^{-A_\eps(t,s)}b(s,X^\eps_s,\xi_{s/\eps})ds
+\dfrac 1{\eps^2}H_\eps(t),
\end{aligned}
\end{equation}
where for any $0\leq s\leq t\leq 1, \eps>0$,
$$A_\eps(t,s):=\dfrac 1{\eps^2}\int_s^t\lambda_\eps(r,X_r^\eps)dr,\;\;\;A_\eps(t)=A_\eps(t,0),$$
$$H_\eps(t):=\sqrt \eps e^{-A_\eps(t)}\int_0^te^{A_\eps(s)}
\sigma_\eps(s,X^\eps_s)dw(s).$$
Therefore, we obtain the formula for $X_t^\eps$ as follows
\begin{equation}\label{qint-1}
\begin{aligned}
X^\eps_t=x_0+x_1\int_0^t e^{-A_\eps(s)}ds+\dfrac 1{\eps^2}\int_0^t\int_0^s e^{-A_\eps(s,r)}b(r,X^\eps_r,\xi_{r/\eps})dr ds+\dfrac 1{\eps^2}\int_0^t H_\eps(s)ds.
\end{aligned}
\end{equation}
Using  an integration by parts formula, we have
\begin{equation}\label{qint}
\begin{aligned}
X^\eps_t=x_0+\int_0^t \frac{b(s,X^\eps_s,\xi_{s/\eps})}{\lambda_\eps\left(s,X^\eps_s\right)}ds+\sqrt\eps\int_0^t \frac{\sigma_\eps(s,X^\eps_s)}{\lambda_\eps\left(s,X^\eps_s\right)}dw(s)
+R_\eps(t),
\end{aligned}
\end{equation}
where
\beq \label{R-def-eq}
\begin{aligned}
R_\eps(t):=&x_1\int_0^t e^{-A_\eps(s)}ds-\frac 1{\lambda_\eps(t,X^\eps_t)}\int_0^t e^{-A_\eps(t,s)}b(s,X^\eps_s,\xi_{s/\eps})ds\\
&-\int_0^t\frac 1{\lambda^2_\eps(s,X^\eps_s)} \left(\int_0^s e^{-A_\eps(s,r)}b(r,X^\eps_r,\xi_{r/\eps})dr\right)
\left(\nabla_s\lambda_\eps(s,X^\eps_s)+\left\langle \nabla_X\lambda_\eps(s,X^\eps_s),p^\eps_s\right\rangle\right)ds\\
&-\dfrac 1{\lambda_\eps(t,X^\eps_t)}H_\eps(t)-\int_0^t \dfrac 1{\lambda^2_\eps(s,X^\eps_s)}H_\eps(s)
\left(\nabla_s\lambda(s,\eps^2X^\eps_s)+\left\langle \nabla_X\lambda_\eps(s,X^\eps_s),p^\eps_s\right\rangle\right)ds\\
=:&\sum_{i=1}^5 R_\eps^{(i)}(t).
\end{aligned}
\eeq

\subsection{Exponential Tightness of $\{X^\eps\}_{\eps>0}$}
\label{sec:exp}
In this section, we investigate the exponential tightness of $\{X^\eps\}_{\eps>0}$ in $\CC([0,1],\R^d)$ by proving \eqref{p-cre-1} and \eqref{p-cre-2}.
In what follows, whenever the estimates involve random variables, they should be understood in the sense of with probability 1 if it is not specified otherwise, which is our convention henceforth.

\subsubsection{Proof of \eqref{p-cre-1}}\label{sec:expbound}
We begin with the following Proposition, whose proof is postponed to the Appendix.

\begin{prop}\label{prop-Heps}
There is a finite constant $\hat C_0$, independent of $\eps$ such that
\begin{equation}\label{eq-37}
\|H_\eps\|\leq \hat C_0\sqrt\eps\norm{w}e^{\hat C_0\sqrt\eps\norm w}.
\end{equation}
As a result, there is a finite constant $\hat C_1$, independent of $\eps$ such that
\begin{equation*}
\|X^\eps\|\leq \hat C_1\Gamma(\hat C_1\sqrt\eps\norm{w}),
\end{equation*}
where
$$
\Gamma(v):=(1+ve^v+v^2e^{2v})(1+v)e^{1+ve^v},\;v\geq 0.
$$
\end{prop}

\begin{proof}[Proof of \eqref{p-cre-1}]
We have from Proposition \ref{prop-Heps} that
\begin{equation}\label{eq-normX}
\|X^\eps\|\leq \hat C_1\Gamma(\hat C_1\sqrt\eps\norm{w}),
\end{equation}
for some finite constant $\hat C_1$ independent of $\eps$ and $\Gamma(\cdot)$ as in Proposition \ref{prop-Heps}.
On the other hand, by the LDP for
 the Brownian motion $w(t)$, we have
\begin{equation}\label{LDPW}
\begin{aligned}
\displaystyle&\lim_{L\to\infty}\limsup_{\eps\to 0}\eps\log\PP\{\sqrt\eps\norm{w}\geq L\}
= -\infty.
\end{aligned}
\end{equation}
Combining \eqref{eq-normX} and \eqref{LDPW}, we obtain
$$\begin{aligned}
\displaystyle\lim_{L\to\infty}&\limsup_{\eps\to 0}\eps\log\PP\left\{\norm{X^\eps}\geq L\right\}
\\&\leq \lim_{L\to\infty}\limsup_{\eps\to 0}\eps\log\PP\Big\{\sqrt\eps\norm{w}\geq M(L,\hat C_1)\Big\}
\\&=-\infty,
\end{aligned}$$
where $M(L,\hat C_1)$
is a constant depending on $L$ and $\hat C_1$ that tends to $\infty$ as $L\to \infty$. Therefore, the proof is complete.
\end{proof}

\subsubsection{Proof of \eqref{p-cre-2}}\label{sub:expcont}
\begin{prop}\label{prop-Hepsts}
For any $\ell>0$, we have
$$
\lim_{\delta\to 0}\limsup_{\eps\to0}\sup_{s\in [0,1]}\eps\log\PP\Big(\sup_{s\leq t\leq s+\delta}\left|H_\eps(t)-H_\eps(s)\right|>\ell\Big)=-\infty.
$$
\end{prop}

\begin{proof}
Let $\ell>0$ be fixed. Let $\delta\in(0,1)$ be fixed but otherwise arbitrary and $\eps\in(0,1)$.
In what follows, we mainly work with $s,t\in [0,1]$, $0\leq t-s\leq \delta$.
Denote
$$\bar A_\eps(t,s):=
\frac 1{\eps^2}\int_s^t \lambda_\eps(r,0)dr,\;
\bar A_\eps(t):=\bar A_\eps(t,0),\;
M_\eps(t,s):=A_\eps(t,s)-\bar A_\eps(t,s),\; M_\eps(t):=M_\eps(t,0).$$
Using property of $\lambda$, it is easily seen that
\begin{equation}\label{eq-Mbar}
|M_\eps(t,s)|\leq \kappa_1\int_s^t |X^\eps_r|dr,
\end{equation}
where $\kappa_1$ is some finite constant, independent of $\eps,t,s$.
By definition of $H_\eps(t)$ in Section \ref{sec:rep}, we rewrite
\begin{equation}\label{prop32-eq1}
H_\eps(t)=\sqrt\eps e^{-M_\eps(t)}h_\eps(t),
\text{ where }
h_\eps(t):=e^{-\bar A_\eps(t)}\int_0^t e^{A_\eps(r)}\sigma_\eps(r,X^\eps_r)dw(r),
\end{equation}
and denote
\begin{equation}\label{eq-def-heps}
\begin{aligned}
h_\eps(t,s):=&h_\eps(t)-h_\eps(s)\\
=&\int_s^t e^{-\bar A_\eps(t,r)+ M_\eps(r)}\sigma_\eps(r,X^\eps_r)dw(r)-(1-e^{-\bar A_\eps(t,s)}) \int_0^s e^{-\bar A_\eps(s,r)+M_\eps(r)}\sigma_\eps(r,X^\eps_r)dw(r).
\end{aligned}
\end{equation}
In the above, we can move $\bar A_\eps(t)$ into the It\^o's integral because of independence of $\bar A_\eps(t)$ and $w(t)$.
Moreover, it is noted that $h_\eps(t)$ (resp. $h_\eps(t,s)$) is $\R^d$-valued, and we will denote by $h_\eps^{(i)}(t)$ (resp. $h_\eps^{(i)}(t,s)$) their $i$-th component, $i=1,\dots,d$.
Using \eqref{prop32-eq1}, we have following decomposition
$$
\begin{aligned}
H_\eps(t)-H_\eps(s)=&\Big(\sqrt\eps e^{-M_\eps(t)}h_\eps(t)-\sqrt\eps e^{-M_\eps(t)}h_\eps(s)\Big)
+\Big(\sqrt\eps e^{-M_\eps(t)}h_\eps(s)-\sqrt\eps e^{-M_\eps(s)}h_\eps(s)\Big)\\
=:& K_\eps^{(1)}(t,s)+K_\eps^{(2)}(t,s).
\end{aligned}
$$

Next we proceed to estimate $K_\eps^{(1)}$  and $K_\eps^{(2)}$.
Although $-\bar A_\eps(t,r)$ and $M_\eps(r)$ are adapted
with respect to filtration generated by the Brownian motion, $\{h_\eps^{(i)}(t,s)\}_{t\geq s}$ is not a martingale with respect to $t$. However, if we frozen $t$, the sequence
$$
\int_s^{t'}e^{-\bar A_\eps(t,r)+ M_\eps(r)}\sigma^{(i)}_\eps(r,X^\eps_r)dw(r)
$$
is a martingale with respect to $t'\in [s,t]$ and has the quadratic variation,
$$
\int_s^{t'} e^{-2\bar A_\eps(t,r)+2 M_\eps(r)}|\sigma^{(i)}_\eps(r,X^\eps_r)|^2dr,
$$
where $\sigma^{(i)}_\eps$ is $i$-th row of $\sigma_\eps$.
Because of \eqref{eq-Mbar},
we have
\begin{equation*}
\begin{aligned}
\int_s^t e^{-2\bar A_\eps(t,r)+2 M_\eps(r)}|\sigma^{(i)}_\eps(r,X^\eps_r)|^2dr
\leq Ce^{2\kappa_1\|X^\eps\|}|t-s|.
\end{aligned}
\end{equation*}
Similarly, using the fact $(1-e^{-u})^2\leq u, \forall u>0$, $\int_0^s e^{-\bar A_\eps(s,r)}dr\leq \int_0^s \frac{-\kappa_0(s-r)}{\eps^2}dr\leq \frac{\eps^2}{\kappa_0}$, one has
$$(1-e^{-\bar A_\eps(t,s)}) \int_0^s e^{-\bar A_\eps(s,r)+M_\eps(r)}\sigma^{(i)}_\eps(r,X^\eps_r)dw(r)$$ is an element of sequence of martingale with quadratic variation bounded by
$Ce^{2\kappa_1\|X^\eps\|}|t-s|$.
Therefore, by applying exponential inequality for martingale (see e.g., \cite[Theorem 7.4, p. 44]{Mao97}) for the above two stochastic integrals, we have from \eqref{eq-def-heps} that for all $s\leq s_1\leq s_2\leq t\leq s+\delta$,
$$
\begin{aligned}
\PP\Bigg\{|h_\eps^{(i)}(s_2,s_1)|&\geq \frac {\delta^{1/8}|s_2-s_1|^{1/8}}{d\sqrt{\eps}}+\frac {1}{Cd\sqrt{\eps}\delta^{1/4}|s_2-s_1|^{3/8}}
Ce^{2\kappa_1\|X^\eps\|}|s_2-s_1|
\Bigg\}\\
& \leq4\exp\left\{\frac {-1}{2d^2C\eps \delta^{1/8}|s_2-s_1|^{1/4}}\right\}.
\end{aligned}
$$
Therefore, one has that for all $s\leq s_1\leq s_2\leq t\leq s+\delta$,
\begin{equation}\label{eq-1112}
\begin{aligned}
\PP\Bigg\{|h_\eps^{(i)}(s_2,s_1)|\geq \Bigg(\frac {\delta^{1/8}}{d\sqrt{\eps}}+\frac {e^{2\kappa_1\|X^\eps\|}\delta^{1/4}}{d\sqrt{\eps}} \Bigg)|s_2-s_1|^{1/8}
\Bigg\} \leq4\exp\left\{\frac {-1}{2d^2C\eps \delta^{1/8}|s_2-s_1|^{1/4}}\right\}.
\end{aligned}
\end{equation}
To proceed, we have following
lemma.
\begin{lm}\label{lm-1111}
	Assume that $Y(t)$ is a continuous stochastic process, that $L$ is a random variable, and that there are constants $\alpha_1$ and $\alpha_2>0$ such that
	\begin{equation}\label{eq-1111}
	\PP(|Y(s_2)-Y(s_1)|\geq L|s_2-s_1|^{1/8})\leq \alpha_1\exp\left\{-\frac{\alpha_2}{|s_2-s_1|^{1/4}}\right\},\forall s_2,s_1\in [0,1].
	\end{equation}
	There are constants $C_1$, $C_2$, and $C_3>0$ such that
	$$
	\PP(\sup_{t\in[0,1]}|Y(t)|\geq C_3L)\leq C_1\alpha_1\exp\left\{-C_2\alpha_2\right\}.
	$$
\end{lm}

Now, we obtain from \eqref{eq-1112} and Lemma \ref{lm-1111}
that
\begin{equation*}
\begin{aligned}
\PP\left\{\sup_{t\in [s,s+\delta]}|h_\eps^{(i)}(t,s)|\geq \frac {\hat C_4\delta^{1/8}}{d\sqrt{\eps }}+\frac {\hat C_4e^{2\kappa_1\|X^\eps\|}\delta^{1/4}}{d\sqrt{\eps}}
\right\}\leq \hat C_2\exp\left\{\frac {-\hat C_3}{\eps \delta^{1/8}}\right\},\;i=1,\dots,d,
\end{aligned}
\end{equation*}
for some positive constants $\hat C_2, \hat C_3, \hat C_4$, independent of $\eps,t,s,\delta$,
which implies that
\begin{equation}\label{eq-heps}
\begin{aligned}
\PP\left\{\sup_{t\in [s,s+\delta]}|h_\eps(t,s)|\geq \frac {\hat C_4\delta^{1/8}}{\sqrt{\eps }}+\frac {\hat C_4e^{2\kappa_1\|X^\eps\|}\delta^{1/4}}{\sqrt{\eps}}
\right\}\leq d\hat C_2\exp\left\{\frac {-\hat C_3}{\eps \delta^{1/8}}\right\}.
\end{aligned}
\end{equation}
Combining \eqref{eq-heps},  the fact that $|K_\eps^{(1)}(t,s)|\leq \sqrt\eps e^{-M(t)}|h_\eps(t,s)|$, and \eqref{eq-Mbar}, one has that for all $s,t\in [0,1]$, $0\leq t-s\leq \delta$,
\begin{equation}\label{eq-A1ts}
\begin{aligned}
\PP\left\{\sup_{t\in[s,s+\delta]}|K_\eps^{(1)}(t,s)|\geq \hat C_4\delta^{1/8}e^{\kappa_1\|X^\eps\|}+\hat C_4e^{3\kappa_1\|X^\eps\|}\delta^{1/4}\right\}\leq d\hat C_2\exp\left\{\frac {-\hat C_3}{\eps \delta^{1/8}}\right\}.
\end{aligned}
\end{equation}
From \eqref{eq-A1ts} and the logarithm equivalence principle  \cite[Lemma 1.2.15]{DZ98}, we have the following estimates
\begin{equation}\label{eq-A1ts-1}
\begin{aligned}
\limsup_{\eps\to0}&\sup_{s\in [0,1]}\eps\log\PP\Big(\sup_{s\leq t\leq s+\delta}\left|K_\eps^{(1)}(t,s)\right|>\ell\Big)\\
\leq&\limsup_{\eps\to0}\sup_{s\in [0,1]}\eps\log\Bigg(\PP\left\{\sup_{s\leq t\leq s+\delta}|K_\eps^{(1)}(t,s)|\geq \hat C_4\delta^{1/8}e^{\kappa_1\|X^\eps\|}+\hat C_4e^{3\kappa_1\|X^\eps\|}\delta^{1/4}\right\}\\
&\hspace{5cm}+\PP\left\{\hat C_4\delta^{1/8}e^{\kappa_1\|X^\eps\|}\geq \frac {\ell}2\right\}
+\PP\left\{\hat C_4e^{3\kappa_1\|X^\eps\|}\delta^{1/4}\geq \frac {\ell}2\right\}\Bigg)\\
=&\limsup_{\eps\to0}\eps\log\Bigg(d\hat C_2\exp\Big\{\frac {-\hat C_3}{\eps \delta^{1/8}}\Big\}\vee\PP\left\{\hat C_4\delta^{1/8}e^{\kappa_1\|X^\eps\|}\geq \frac {\ell}2\right\}\vee\PP\left\{\hat C_4e^{3\kappa_1\|X^\eps\|}\delta^{1/4}\geq \frac {\ell}2\right\}\Bigg).
\end{aligned}
\end{equation}
Now, letting $\delta\to0$,
it is  seen that
\begin{equation}\label{eq-A1ts-2}
\lim_{\delta\to 0}\limsup_{\eps\to0}\eps\log d\hat C_2\exp\Big\{\frac {-\hat C_3}{\eps \delta^{1/8}}\Big\}=-\infty,
\end{equation}
and from \eqref{p-cre-1} that
\begin{equation}\label{eq-A1ts-3}
\begin{aligned}
\lim_{\delta\to 0}&\limsup_{\eps\to0}\eps\log\PP\left\{\hat C_4\delta^{1/8}e^{\kappa_1\|X^\eps\|}\geq \frac {\ell}2\right\}\\
&=\limsup_{\delta\to 0}\limsup_{\eps\to0}\eps\log\PP\left\{\hat C_4e^{3\kappa_1\|X^\eps\|}\delta^{1/4}\geq \frac {\ell}2\right\}\Bigg)\\
&=-\infty.
\end{aligned}
\end{equation}
Combining \eqref{eq-A1ts-1}, \eqref{eq-A1ts-2}, and \eqref{eq-A1ts-3} leads to that
\beq\label{eq-A1ts-4}
\lim_{\delta\to 0}\limsup_{\eps\to0}\sup_{s\in [0,1]}\eps\log\PP\Big(\sup_{s\leq t\leq s+\delta}\left|K_\eps^{(1)}(t,s)\right|>\ell\Big)=-\infty.
\eeq

Next, we prove similar results for $K_\eps^{(2)}(t,s)$.
The mean value theorem and
\eqref{eq-Mbar}
imply that
$$
\left |e^{-M_\eps(t)}-e^{-M_\eps(s)}\right|
\leq \kappa_1\|X^\eps\|e^{\kappa_1\|X^\eps\|}|t-s|.
$$
Therefore, for all $s,t\in [0,1]$, $s\leq t\leq s+ \delta$
\beq\label{eq-A2ts}
\sup_{t\in[s,s+\delta]}\left|K_\eps^{(2)}(t,s)\right|\leq  \kappa_1\sqrt\eps\|X^\eps\|e^{\kappa_1\|X^\eps\|}\delta\left|h_\eps(s)\right|.
\eeq
Since
$$
\begin{aligned}
\int_0^s e^{-2\bar A_\eps(s,r)+2 M_\eps(r)}|\sigma^{(i)}(r,\eps^2X^\eps_r)|^2dr
\leq Ce^{2\kappa_1\|X^\eps\|}\int_0^se^{-\frac{2\kappa_0(s-r)}{\eps^2}}dr
\leq \hat C_5\eps^2e^{2\kappa_1\|X^\eps\|}, i=1,\dots,d,
\end{aligned}
$$
where $\hat C_5$ is a finite constant depending only on $\kappa_0$ and function $\sigma$;
by exponential inequality for martingale (see e.g., \cite[Theorem 7.4, p. 44]{Mao97}) again and similar process of getting \eqref{eq-heps},
we have
\beq\label{eq-A2ts-2}
\begin{aligned}
	\PP\left\{|h_\eps(s)|\geq \frac {1}{\sqrt{\eps}}+\frac {1}{\hat C_5\eps^2\sqrt{\eps}} \hat C_5\eps^2e^{2\kappa_1\|X^\eps\|}\right\}\leq2d\exp\left\{-\frac {1}{d\sqrt{\eps}}\cdot \frac {2}{d\hat C_5\sqrt{\eps}\eps^2}\right\} \leq2d\exp\left\{\frac {-2}{\hat C_5d^2\eps^3}\right\}.
\end{aligned}
\eeq
Combining \eqref{eq-A2ts} and \eqref{eq-A2ts-2} enables us to obtain that
\beq
\begin{aligned}
	\PP\left\{\sup_{t\in[s,s+\delta]}|K_\eps^{(2)}(t,s)|\geq \kappa_1\|X^\eps\|e^{\kappa_1\|X^\eps\|}\delta+\kappa_1\|X^\eps\|e^{3\kappa_1\|X^\eps\|}\delta\right\} \leq2d\exp\left\{\frac {-2}{\hat C_5d^2\eps^3}\right\}.
\end{aligned}
\eeq
Therefore, by a similar argument for
getting \eqref{eq-A1ts-4}, we obtain
$$
\lim_{\delta\to 0}\limsup_{\eps\to0}\sup_{s\in [0,1]}\eps\log\PP\Big(\sup_{s\leq t\leq s+\delta}\left|K_\eps^{(2)}(t,s)\right|>\ell\Big)=-\infty.
$$
As a consequence, we have
$$
\lim_{\delta\to 0}\limsup_{\eps\to0}\sup_{s\in [0,1]}\eps\log\PP\Big(\sup_{s\leq t\leq s+\delta}\left|H_\eps(t)-H_\eps(s)\right|>\ell\Big)=-\infty.
$$
\end{proof}

\begin{proof}[Proof of \eqref{p-cre-2}]
Let $\ell>0$ be fixed. Let $\delta\in(0,1)$ be fixed but otherwise arbitrary and $\eps\in(0,1)$.
In what follows, we mainly work with $s,t\in [0,1]$, $0\leq t-s\leq \delta$.
Because of \eqref{qint}, we have
\beq\label{eq-cont-0}
\begin{aligned}
	|X^\eps_t-X^\eps_s|\leq &\left|\int_s^t \frac{b(r,X^\eps_r,\xi_{r/\eps})}{\lambda_\eps\left(r,X^\eps_r\right)}dr\right|+\sqrt\eps\left|\int_s^t \frac{\sigma_\eps(r,X^\eps_r)}{\lambda_\eps\left(r,X^\eps_r\right)}dw(r)\right|+|R_\eps(t)-R_\eps(s)|.
\end{aligned}
\eeq
Since
$$
\left|\int_s^t \frac{b(r,X^\eps_r,\xi_{r/\eps})}{\lambda_\eps\left(r,X^\eps_r\right)}dr\right|\leq C(1+\|X^\eps\|)|t-s|,
$$
it is easily seen from \eqref{p-cre-1} that
\beq\label{eq-cont-1}
\lim_{\delta\to 0}\limsup_{\eps\to0}\sup_{s\in [0,1]}\eps\log\PP\left(\sup_{s\leq t\leq s+\delta}\left|\int_s^t \frac{b(r,X^\eps_r,\xi_{r/\eps})}{\lambda_\eps\left(r,X^\eps_r\right)}dr\right|>\ell\right)=-\infty.
\eeq
By our assumptions on $\lambda$ and $\sigma$, the coefficient of the diffusion
$
\sqrt\eps\int_s^t \frac{\sigma_\eps(r,X^\eps_r)}{\lambda_\eps\left(r,X^\eps_r\right)}dw(r)
$
is uniformly bounded.
Therefore, the Bernstein inequality \cite[pp. 153-154]{RY94} yields
\beq\label{eq-cont-2}
\lim_{\delta\to 0}\limsup_{\eps\to0}\sup_{s\in [0,1]}\eps\log\PP\left(\sup_{s\leq t\leq s+\delta}\sqrt\eps\left|\int_s^t \frac{\sigma_\eps(r,X^\eps_r)}{\lambda_\eps\left(r,X^\eps_r\right)}dw(r)\right|>\ell\right)=-\infty.
\eeq

Next, we  consider the term $|R_\eps(t)-R_\eps(s)|$.
We have
\beq\label{eq-R0}
|R_\eps(t)-R_\eps(s)|\leq \sum_{i=1}^5|R_\eps^{(i)}(t)-R_\eps^{(i)}(s)|.
\eeq
First, it is clear that
\beq\label{eq-R1}
\lim_{\delta\to 0}\limsup_{\eps\to0}\sup_{s\in [0,1]}\eps\log\PP\left(\sup_{s\leq t\leq s+\delta}\left|R_\eps^{(1)}(t)-R_\eps^{(1)}(s)\right|>\ell\right)=-\infty.
\eeq
Second, we have
$$
\begin{aligned}
|R_\eps^{(2)}(t)-R_\eps^{(2)}(s)|\leq& \frac{1}{\lambda_\eps(t,X^\eps_t)}\left|\int_s^t e^{-A_\eps(t,r)}b(r,X^\eps_r,\xi_{r/\eps})dr\right|\\
&+\frac{|\lambda_\eps(t,X^\eps_t)-\lambda_\eps(s,X^\eps_s)|}{\lambda_\eps(t,X^\eps_t)\lambda_\eps(s,X^\eps_s)}\left|\int_0^s e^{-A_\eps(t,r)}b(r,X^\eps_r,\xi_{r/\eps})dr\right|\\
\leq & C(1+\|X^\eps\|)|t-s|+C\left(1+\|X^\eps\|\right)(|t-s|+\eps^2\|X^\eps\|),
\end{aligned}
$$
and then, \eqref{p-cre-1} gives us that
\beq\label{eq-R2}
\lim_{\delta\to 0}\limsup_{\eps\to0}\sup_{s\in [0,1]}\eps\log\PP\left(\sup_{s\leq t\leq s+\delta}\left|R_\eps^{(2)}(t)-R_\eps^{(2)}(s)\right|>\ell\right)=-\infty.
\eeq
Third, one has from definition of $R_\eps^{(3)}$, property of $\lambda_\eps(\cdot,\cdot)$ and \eqref{eq-2.8-1} that
$$
\begin{aligned}
&|R_\eps^{(3)}(t)-R_\eps^{(3)}(s)|\\
&\quad=\left|\int_s^t\frac 1{\lambda^2_\eps(r,X^\eps_r)} \left(\int_0^r e^{-A_\eps(r,r')}b(r',X^\eps_{r'},\xi_{r'/\eps})dr'\right)\left(\nabla_s\lambda_\eps(r,X^\eps_r)+\left\langle \nabla_X\lambda_\eps(r,X^\eps_r),p^\eps_r\right\rangle\right)dr\right|\\
&\quad\leq C (1+\|X^\eps\|)(1+\|H_\eps\|)|t-s|,
\end{aligned}
$$
which combined with Proposition \ref{prop-Heps} and the LDP of $\sqrt\eps\|w\|$ implies that
\beq\label{eq-R3}
\lim_{\delta\to 0}\limsup_{\eps\to0}\sup_{s\in [0,1]}\eps\log\PP\left(\sup_{s\leq t\leq s+\delta}\left|R_\eps^{(3)}(t)-R_\eps^{(3)}(s)\right|>\ell\right)=-\infty.
\eeq
Next, we have that
$$
\begin{aligned}
|R_\eps^{(4)}(t)-R_\eps^{(4)}(s)|\leq& \frac{1}{\lambda_\eps(t,X^\eps_t)}|H_\eps(t)-H_\eps(s)|+\frac{|\lambda_\eps(t,X^\eps_t)-\lambda_\eps(s,X^\eps_s)|}{\lambda_\eps(t,X^\eps_t)\lambda_\eps(s,X^\eps_s)}|H_\eps(s)|\\
\leq &C|H_\eps(t)-H_\eps(s)|+C\|H_\eps\|(|t-s|+\eps^2\|X^\eps\|).
\end{aligned}
$$
Hence, using Proposition \ref{prop-Hepsts} to take care of the term $|H_\eps(t)-H_\eps(s)|$, and using \eqref{p-cre-1} and Proposition \ref{prop-Heps} to take care of the term $\|H_\eps\|(|t-s|+\eps^2\|X^\eps\|)$, we can obtain
\beq\label{eq-R4}
\lim_{\delta\to 0}\limsup_{\eps\to0}\sup_{s\in [0,1]}\eps\log\PP\left(\sup_{s\leq t\leq s+\delta}\left|R_\eps^{(4)}(t)-R_\eps^{(4)}(s)\right|>\ell\right)=-\infty.
\eeq
Finally, $R_\eps^{(5)}$ is handled similarly as that of  $R_\eps^{(3)}$. Since
$$
|R_\eps^{(5)}(t)-R_\eps^{(5)}(s)|\leq C \|H_\eps\|(1+\|H_\eps\|)|t-s|,
$$
using  Proposition \ref{prop-Heps} and the LDP of $\sqrt\eps\|w\|$, one has
\beq\label{eq-R5}
\lim_{\delta\to 0}\limsup_{\eps\to0}\sup_{s\in [0,1]}\eps\log\PP\left(\sup_{s\leq t\leq s+\delta}\left|R_\eps^{(5)}(t)-R_\eps^{(5)}(s)\right|>\ell\right)=-\infty.
\eeq
We obtain from \eqref{eq-R0}-\eqref{eq-R5} that
\beq\label{eq-cont-3}
\lim_{\delta\to 0}\limsup_{\eps\to0}\sup_{s\in [0,1]}\eps\log\PP\left(\sup_{s\leq t\leq s+\delta}\left|R_\eps(t)-R_\eps(s)\right|>\ell\right)=-\infty.
\eeq
Combining \eqref{eq-cont-0}, \eqref{eq-cont-1}, \eqref{eq-cont-2}, and \eqref{eq-cont-3}, we get \eqref{p-cre-2}.

\end{proof}

\subsection{Local LDP of $\{X^\eps\}_{\eps>0}$}\label{sec:local}
We begin this section with the following Proposition, which provides a kind of ``exponential equivalence property" of $X^\eps$ and $q^\eps$ in the ``local sense".

\begin{prop}\label{close}
For any
$\theta>0$, $N>0$, and $\varphi \in \CC([0,1],\R^d)$ that is absolutely continuous, there exist $\bar \theta_1,\bar \theta_2>0$, independent of $\eps$ and $\eps_0>0$ such that for any $\eps<\eps_0$,
\begin{equation}\label{Cl-1}
\PP\left\{\|q^\eps-\varphi\|<\theta\right\}\geq \PP\left\{\|X^\eps-\varphi\|<\bar\theta_1\right\}
-\exp\left\{-\dfrac{N}{\eps}\right\},\end{equation}
\begin{equation}\label{Cl-2}\PP\left\{\|X^\eps-\varphi\|<\theta\right\}\geq \PP\left\{\|q^\eps-\varphi\|<\bar\theta_2\right\}
-\exp\left\{-\dfrac{N}{\eps}\right\}.\end{equation}
Moreover, $\bar\theta_1$ and $\bar\theta_2$ are
to be specified
later $($see \eqref{eq-5}$)$.

\end{prop}

\begin{proof}
For each $\varphi(\cdot)\in \CC([0,1],\R^d)$ that
is absolutely continuous, we denote by $X^{\eps,\varphi}_t$ the solution of
\begin{equation}\label{3.1}
\begin{cases}
\eps^2\ddot{X}^{\eps,\varphi}_t=b(t,\varphi_t,\xi_{t/\eps})-\lambda_\eps(t,\varphi_t)\dot{X}^{\eps,\varphi}_t+\sqrt{\eps}\sigma_\eps(t,\varphi_t)\dot w(t),\\
X^{\eps,\varphi}_0=x_0\in\R^d;\;\;\;\;\;\dot{X}^\varphi_\eps(0)=x_1\in \R^d,
\end{cases}
\end{equation}
and by $q^{\eps,\varphi}_t$ the solution of
\begin{equation}\label{3.2}
\begin{cases}
\dot{q}^{\eps,\varphi}_t=\dfrac{b(t,\varphi_t,\xi_{t/\eps})}{\lambda_\eps(t,\varphi_t)}+\sqrt{\eps}\dfrac{\sigma_\eps(t,\varphi_t)}{\lambda_\eps(t,\varphi_t)}\dot {w}(t),\\
q^{\eps,\varphi}_0=x_0\in\R^d.
\end{cases}
\end{equation}
Recall that
$C$ is a generic positive constant whose value may change for different appearances. The constant $C$ may depend on initial values $x_0,x_1$ and coefficients $b,\lambda,\sigma$, but is independent of $\eps$ and  $\varphi$. Note that the time variable $t$ is always assumed to be
in $[0,1].$
Now, it is readily seen that
\begin{equation}\label{est}
\abs{X^\eps_t-\varphi_t}\leq \abs{X^\eps_t-X^{\eps,\varphi}_t}+\abs{X^{\eps,\varphi}_t-q^{\eps,\varphi}_t}+\abs{q^{\eps,\varphi}_t-q^\eps_t}+\abs{q^\eps_t-\varphi_t}.
\end{equation}

\noindent\textbf{Step 1: Estimate of $\abs{X^\eps_t-X^{\eps,\varphi}_t}$.}
The following  decomposition will be used often in the proof
$$
u(t)v(t)-u(s)v(s)=u(t)(v(t)-v(s))+v(s)(u(t)-u(s)).
$$
As in Section \ref{sec:rep}, we have
\begin{equation}\label{3.4}
X^{\eps,\varphi}_t=x_0+x_1\int_0^t e^{-A^{\varphi}_\eps(s)}ds+\dfrac 1{\eps^2}\int_0^t\int_0^s e^{-A^\varphi_\eps(s,r)}b(r,\varphi_r,\xi_{r/\eps})dr ds+\dfrac 1{\eps^2}\int_0^t H^\varphi_\eps(s)ds,
\end{equation}
where
$$A_\eps^\varphi(t,s)=\dfrac 1{\eps^2}\int_s^t\lambda_\eps(r,\varphi_r)dr,\quad A_\eps^\varphi(t)=A_\eps^\varphi(t,0),$$
$$H_\eps^\varphi(t)=\sqrt{\eps} e^{-A_\eps^\varphi(t)}\int_0^t e^{A^\varphi_\eps(r)}\sigma_\eps(r,\varphi_r)dw(r).$$
Thus, using integration by parts, we have
\begin{equation}\label{qint-varphi}
\begin{aligned}
X^{\eps,\varphi}_t=x_0+\int_0^t \frac{b(s,\varphi_s,\xi_{s/\eps})}{\lambda_\eps\left(s,\varphi_s\right)}ds+\sqrt\eps\int_0^t \frac{\sigma_\eps(s,\varphi_s)}{\lambda_\eps\left(s,\varphi_s\right)}dw(s)+R_\eps^\varphi(t),
\end{aligned}
\end{equation}
where
$$
\begin{aligned}
R_\eps^\varphi(t):=&x_1\int_0^t e^{-A_\eps^\varphi(s)}ds-\frac 1{\lambda_\eps(t,\varphi_t)}\int_0^t e^{-A_\eps^\varphi(t,s)}b(s,\varphi_s,\xi_{s/\eps})ds\\
&-\int_0^t\frac 1{\lambda^2_\eps(s,\varphi_s)} \left(\int_0^s e^{-A_\eps^\varphi(s,r)}b(r,\varphi_r,\xi_{r/\eps})dr\right)
\left(\nabla_s\lambda_\eps(s,\varphi_s)+\left\langle \nabla_X\lambda_\eps(s,\varphi_s),\dot{\varphi}_s\right\rangle\right)ds\\
&-\dfrac 1{\lambda_\eps(t,\varphi_t)}H_\eps^\varphi(t)-\int_0^t \dfrac 1{\lambda^2_\eps(s,\varphi_s)}H_\eps^\varphi(s)
\left(\nabla_s\lambda_\eps(s,\varphi_s)+\left\langle \nabla_X\lambda_\eps(s,\varphi_s),\dot{\varphi}_s\right\rangle\right)ds\\
=:&\sum_{i=1}^5 R_\eps^{\varphi,(i)}(t).
\end{aligned}
$$
Therefore, we obtain from \eqref{qint-1}, \eqref{3.4}, and an integration by parts formula (applied to the stochastic integral only) that
\beq\label{eq-XXphi-0}
\begin{aligned}
|X^\eps_t-X^{\varphi,\eps}_t|\leq& \left|\dfrac 1{\eps^2}\int_0^t\int_0^s \left(e^{-A_\eps(s,r)}b(r,X^\eps_r,\xi_{r/\eps})-e^{-A^\varphi_\eps(s,r)}b(r,\varphi_r,\xi_{r/\eps})\right)dr ds\right|\\
&+|D_1^\varphi(t)|+
\left|R_\eps^{(1)}(t)-R_\eps^{\varphi,(1)}(t)\right|
+\left|R_\eps^{(4)}(t)-R_\eps^{\varphi,(4)}(t)\right|
+\left|R_\eps^{(5)}(t)-R_\eps^{\varphi,(5)}(t)\right|,
\end{aligned}
\eeq
where
$$
D_1^\varphi(t):=\sqrt\eps\int_0^t\left( \frac{\sigma_\eps(s,X^\eps_s)}{\lambda_\eps\left(s,X^\eps_s\right)}-\frac{\sigma_\eps(s,\varphi_s)}{\lambda_\eps\left(s,\varphi_s\right)}\right)dw(s).
$$
From the fact $A_\eps(s),A_\eps^\varphi(s)\geq \frac{\kappa_0 s}{\eps^2}$,
and the property of $\lambda$,
we can obtain that
\begin{equation}\label{3.5}
\begin{aligned}
\abs{e^{-A_\eps(s)}-e^{-A_\eps^\varphi(s)}}
\leq&
C e^{\frac{-\kappa_0 s}{\eps^2}}\cdot\frac{1}{\eps^2}\int_0^s\eps^2\abs{X^\eps_r-\varphi_r}dr=Ce^{\frac{-\kappa_0 s}{\eps^2}}\int_0^s\abs{X^\eps_r-\varphi_r}dr.
\end{aligned}
\end{equation}
Therefore,
we obtain from the Lipschitz property of the coefficients and \eqref{3.5} that
\begin{equation}\label{3.6}
\begin{aligned}
\int_0^s &\abs{e^{-A_\eps(s,r)}b(r,X^\eps_r,\xi_{r/\eps})-e^{-A^\varphi_\eps(s,r)}b(r,\varphi_r,\xi_{r/\eps})}dr
\\&\leq \int_0^s e^{-A_\eps(s,r)}\abs{b(r,X^\eps_r,\xi_{r/\eps})-b(r,\varphi_r,\xi_{r/\eps})}dr
+\int_0^s \abs{b(r,\varphi_r,\xi_{r/\eps})}\abs{e^{-A_\eps(s,r)}-
e^{-A^\varphi_\eps(s,r)}}dr\\
&\leq C(1+\|\varphi\|)\eps^2\sup_{0\leq r\leq s}\abs{X^\eps_r-\varphi_r}.
\end{aligned}
\end{equation}
A consequence of \eqref{3.5} is that
\beq\label{eq-XXphi-1}
\left|R_\eps^{(1)}(t)-R_\eps^{\varphi,(1)}(t)\right|\leq C\eps^2\sup_{s\in[0,t]}|X^\eps_s-\varphi_s|,
\eeq
and
by
\eqref{3.6}
\beq\label{eq-XXphi-2}
\left|\dfrac 1{\eps^2}\int_0^t\int_0^s \left(e^{-A_\eps(s,r)}b(r,X^\eps_r,\xi_{r/\eps})-e^{-A^\varphi_\eps(s,r)}b(r,\varphi_r,\xi_{r/\eps})\right)dr ds\right|\leq C(1+\|\varphi\|)\int_0^t\sup_{0\leq r\leq s}|X^\eps_r-\varphi_r|ds.
\eeq

It is well-known that in a compact interval, an absolutely continuous function $\varphi$ is also of bounded variation.
Moreover, $\varphi$ is differentiable almost everywhere and if we denote
 the derivative by $\dot\varphi$,
 then $\dot\varphi$ is integrable.
 Hence, the argument in the proof of Proposition \ref{prop-Heps} enables us to conclude that \eqref{bypart1}-\eqref{divide}
are valid for the function $\varphi(\cdot)$.
 As a result, by using \eqref{bypart1} for $e^{A_\eps(s)}\sigma_\eps(s,X^\eps_s)$ and $w(s)$, and $e^{A_\eps^\varphi(s)}\sigma_\eps(s,\varphi_s)$ and $w(s)$ to estimate pathwise stochastic integrals $$\int_0^te^{A_\eps(s)}\sigma_\eps(s,X^\eps_s)dw(s)\ \hbox{ and }\ \int_0^te^{A_\eps^\varphi(s)}\sigma_\eps(s,\varphi_s)dw(s),$$ we
 obtain
\beq
\begin{aligned}
&\!\!\! H_\eps(t)-H^\varphi_\eps(t)\\
&\! =\sqrt\eps\left(\sigma_\eps(t,X^\eps_t)-\sigma_\eps(t,\varphi_t)\right)w(t)\\
&\ -\sqrt\eps\int_0^t w(s)\left(e^{-A_\eps(t,s)}-e^{-A_\eps^\varphi(t,s)}\right)\left[\dfrac{\lambda_\eps(s,X^\eps_s)}{\eps^2}\sigma_\eps(s,X^\eps_s)+\nabla_s\sigma_\eps(s,X^\eps_s)+\nabla_X\sigma_X(s, X^\eps_s)p^\eps_s\right]ds\\
&\ +\sqrt\eps\int_0^t w(s)e^{-A_\eps^\varphi(t,s)}\Bigg[\dfrac{\lambda_\eps(s,\varphi_s)}{\eps^2}\sigma_\eps(s,X^\eps_s)
-\dfrac{\lambda_\eps(s,X^\eps_s)}{\eps^2}\sigma(s,\eps^2X^\eps_s)
\Bigg]ds\\
&\ +\sqrt\eps\int_0^t w(s)e^{-A_\eps^\varphi(t,s)}\Bigg[
\dfrac{\lambda_\eps(s,\varphi_s)}{\eps^2}\sigma_\eps(s,\varphi_s)-\dfrac{\lambda_\eps(s,\varphi_s)}{\eps^2}\sigma_\eps(s,X^\eps_s)\\
&\hspace{4cm}\nabla_s\sigma_\eps(s,\varphi_s)+\nabla_X\sigma_\eps(s,\varphi_s)\dot{\varphi}_s
-\nabla_s\sigma_\eps(s,X^\eps_s)-\nabla_X\sigma_\eps(s, X^\eps_s)p^\eps_s\Bigg]ds\\
&\! =:\sum_{i=1}^4 B_{\eps}^{\varphi,(i)}(t).
\end{aligned}
\eeq
Note that
\beq\label{eq-B1}
|B_{\eps}^{\varphi,(1)}(t)|\leq C \sqrt\eps\eps^2\|w\||X^\eps_t-\varphi_t|.
\eeq
Using \eqref{eq-2.8-1} and \eqref{3.5}, we have
\beq\label{eq-B2-1}
\begin{aligned}
|B_{\eps}^{\varphi,(2)}(t)|\leq& C\sqrt\eps\|w\|\left(\frac 1{\eps^2}+1+\|H_\eps\|\right)\int_0^t \left|e^{-A_\eps(t,s)}-e^{-A_\eps^\varphi(t,s)}\right|ds\\
\leq &C\sqrt\eps\|w\|\left(\frac 1{\eps^2}+\|H_\eps\|\right)\sup_{s\in[0,t]}|X^\eps_s-\varphi_s|\int_0^t e^{\frac{-\kappa_0(t-s)}{\eps^2}}(t-s)ds.
\end{aligned}
\eeq
A change of variable  leads to
\begin{equation}\label{eq-B2-2}
\int_0^t \exp\left\{\frac{-\kappa_0 s}{\eps^2}\right\}\cdot\frac{s}{\eps^2}ds=\eps^2\int_0^{\frac t{\eps^2}}e^{-\kappa_0 r}rdr
\leq C\eps^2.
\end{equation}
Combining \eqref{eq-B2-1} and \eqref{eq-B2-2} implies that
\beq\label{eq-B2}
|B_{\eps}^{\varphi,(2)}(t)|\leq C\sqrt\eps\|w\|\left(\eps^2+\eps^4\|H_\eps\|\right)\sup_{s\in[0,t]}|X^\eps_s-\varphi_s|.
\eeq
Next, it is readily seen that
\beq\label{eq-B3}
|B_{\eps}^{\varphi,(3)}(t)|\leq C\sqrt\eps\|w\|\int_0^t|X^\eps_s-\varphi_s|ds.
\eeq
On the other hand, we have
\beq\label{eq-B4-0}
\begin{aligned}
B_\eps^{\varphi,(4)}(t)=&\sqrt\eps e^{-A_\eps^\varphi(t)}\int_0^t w(s)d e^{A_\eps^\varphi(s)}\left(\sigma_\eps(s,\varphi_s)-\sigma_\eps(s,X^\eps_s)\right)\\
=& \sqrt\eps w(t)\left(\sigma_\eps(t,\varphi_t)-\sigma_\eps(t,X^\eps_t)\right)-\sqrt\eps\int_0^te^{-A_\eps^\varphi(t,s)}\left(\sigma_\eps(s,\varphi_s)-\sigma_\eps(s,X^\eps_s)\right)dw(s).
\end{aligned}
\eeq
Thus, we obtain from \eqref{eq-B4-0} that
\beq\label{eq-B4}
\begin{aligned}
|B_\eps^{\varphi,(4)}(t)|\leq C \sqrt\eps\eps^2\|w\||X^\eps_t-\varphi_t|+|D^\varphi_2(t)|,
\end{aligned}
\eeq
where
$$
D^\varphi_2(t):=\sqrt\eps
\int_0^te^{-A_\eps^\varphi(t,s)}
\left(\sigma_\eps(s,\varphi_s)-\sigma_\eps(s,X^\eps_s)\right)dw(s).
$$
Combining \eqref{eq-B1}, \eqref{eq-B2}, \eqref{eq-B3}, and \eqref{eq-B4} implies that
\begin{equation}\label{eq-HxHphi}
\begin{aligned}
|H_\eps(t)-H^\varphi_\eps(t)|\leq& C\sqrt\eps\eps^2\|w\|\left(1+\eps^2\|H_\eps\|\right)\sup_{s\in[0,t]}|X^\eps_s-\varphi_s|\\
&+C\sqrt\eps\|w\|\int_0^t|X^\eps_s-\varphi_s|ds+C|D^\varphi_2(t)|.
\end{aligned}
\end{equation}
Therefore, a standard calculation allows us to obtain that
\beq\label{eq-XXphi-3}
\begin{aligned}
\left|R_\eps^{(4)}(t)-R_\eps^{\varphi,(4)}(t)\right|\leq &\left|\left(\dfrac 1{\lambda_\eps(t,X^\eps_t)}-\dfrac 1{\lambda_\eps(t,\varphi_t)}\right)H_\eps^\varphi(t)+\dfrac 1{\lambda_\eps(t,X^\eps_t)}(H_\eps(t)-H_\eps^\varphi(t)) \right|\\
\leq&C\eps^2\|H_\eps^\varphi\||X^\eps_t-\varphi_t|+ C\sqrt\eps\eps^2\|w\|\left(1+\eps^2\|H_\eps\|\right)\sup_{s\in[0,t]}|X^\eps_s-\varphi_s|\\
&+C\sqrt\eps\|w\|\int_0^t|X^\eps_s-\varphi_s|ds+C|D^\varphi_2(t)|.
\end{aligned}
\eeq
Next, we have
$$
\begin{aligned}
R_\eps^{(5)}(t)-R_\eps^{\varphi,(5)}(t)=&\int_0^t \left(\dfrac {H_\eps(s)
\nabla_s\lambda_\eps(s,X^\eps_s)}{\lambda^2_\eps(s,X^\eps_s)}-\dfrac {H_\eps^\varphi(s)
\nabla_s\lambda_\eps(s,\varphi_s)}{\lambda^2_\eps(s,\varphi_s)}\right)ds\\
&+\int_0^t\left(\frac{H_\eps(s)\left\langle \nabla_X\lambda_\eps(s,X^\eps_s),p^\eps_s\right\rangle}{\lambda^2_\eps(s,X^\eps_s)}-\frac{H_\eps^\varphi(s)\left\langle \nabla_X\lambda_\eps(s,\varphi_s),\dot{\varphi}_s\right\rangle}{\lambda^2_\eps(s,\varphi_s)}\right)ds\\
=:&B_{\eps}^{\varphi,(5)}(t)+B_\eps^{\varphi,(6)}(t).
\end{aligned}
$$
It can be seen that
$$
|B_{\eps}^{\varphi,(5)}(t)|\leq C\int_0^t |H_\eps(s)-H_\eps^\varphi(s)|ds+C\eps^2\|H_\eps^\varphi\|\int_0^t |X^\eps_s-\varphi_s|ds.
$$
On the other hand, using \eqref{eq-2.8-1}, we get
$$
|B_{\eps}^{\varphi,(6)}(t)|\leq C(1+\|H_\eps\|)\int_0^t |H_\eps(s)-H_\eps^\varphi(s)|ds+C\Big[1+\eps^2\int_0^1 |\dot\varphi_s|ds+\|H_\eps\|\Big]\|H_\eps^\varphi\|.
$$
These equations imply
\beq\label{eq-XXphi-4}
\begin{aligned}
|R_\eps^{(5)}(t)-R_\eps^{\varphi,(5)}(t)|\leq& C(1+\|H_\eps\|)\int_0^t |H_\eps(s)-H_\eps^\varphi(s)|ds+C\eps^2\|H_\eps^\varphi\|\int_0^t |X^\eps_s-\varphi_s|ds\\
&+C\Big[1+\eps^2\int_0^1 |\dot\varphi_s|ds+\|H_\eps\|\Big]\|H_\eps^\varphi\|.
\end{aligned}
\eeq
Hence, by combining \eqref{eq-XXphi-0}, \eqref{eq-XXphi-1}, \eqref{eq-XXphi-2}, \eqref{eq-XXphi-3}, \eqref{eq-XXphi-4}, and \eqref{eq-HxHphi},
we obtain
\begin{equation}\label{eq-XXphi}
\barray\disp
\abs{X^\eps_t-X^{\eps,\varphi}_t}\ad\leq  C\eps^2(1+\sqrt\eps\|w\|)\left(1+\|H_\eps\|\right)^2(1+\|H_\eps^\varphi\|)\sup_{s\in[0,t]}|X^\eps_s-\varphi_s|\\
\aad \ +C(1+\|\varphi\|)\left[(1+\sqrt\eps\|w\|)(1+\|H_\eps\|)+\eps^2\|H_\eps^\varphi\|\right]\int_0^t\sup_{r\in[0,s]}|X^\eps_r-\varphi_r|ds\\
\aad \
+C\Big[1+\eps^2\int_0^1 |\dot\varphi_s|ds+\|H_\eps\|\Big]\|H_\eps^\varphi\|
+C\sup_{r\in [0,1]}\left(|D_1^\varphi(r)|+|D^\varphi_2(r)|\right).
\earray
\end{equation}

\noindent \textbf{Step 2: Estimate of $\abs{X^{\eps,\varphi}_t-q^{\eps,\varphi}_t}$.}
In views of \eqref{3.1} and \eqref{3.2}, we have
\begin{equation*}
\begin{aligned}
\eps^2\ddot{X}^{\eps,\varphi}_t=-\lambda_\eps(t,\varphi_t)(\dot{X}^{\eps,\varphi}_t-\dot{q}^{\eps,\varphi}_t).
\end{aligned}
\end{equation*}
Therefore,
\begin{equation}\label{3.7'}
\begin{aligned}
\abs{X^{\eps,\varphi}_t-q^{\eps,\varphi}_t}=\eps^2 \abs{\int_0^t \dfrac{\ddot{X}^{\eps,\varphi}_s}{\lambda_\eps(s,\varphi_s)}ds} \leq \dfrac {\eps^2}{\kappa_0 }\abs{\int_0^t \ddot{X}^{\eps,\varphi}_sds}=\dfrac {\eps^2}{\kappa_0 }\abs{p^{\eps,\varphi}_t-x_1},
\end{aligned}
\end{equation}
where $p^{\eps,\varphi}_t$ is derivative of $X^{\eps,\varphi}_t$ defined similarly to $p^\eps_t$.
As a consequence of \eqref{3.7'} and \eqref{eq-2.8-1}, we have
\begin{equation}\label{eq-Xphiqphi}
\begin{aligned}
\abs{q^{\eps,\varphi}_t-X^{\eps,\varphi}_t}&\leq C(\eps^2+\norm {H_\eps^\varphi}).
\end{aligned}
\end{equation}

\noindent \textbf{Step 3: Estimate of $\abs{q^{\eps,\varphi}_t-q^\eps_t}$.}
Note that
\eqref{3.1} and \eqref{3.2} imply
\begin{equation*}
\begin{cases}
\dot{q}^{\eps,\varphi}_t-\dot{q}^\eps_t=
\left(\dfrac{b(t,\varphi_t,\xi_{t/\eps})}{\lambda_\eps(t,\varphi_t)}
-\dfrac{b(t,q^{\eps}_t,\xi_{t/\eps})}{\lambda_\eps(t,q_t^\eps)}\right)
+\sqrt{\eps}\left(\dfrac{\sigma_\eps(t,\varphi_t)}{\lambda_\eps(t,\varphi_t)}-\dfrac{\sigma_\eps(t,q_t^\eps)}{\lambda_\eps(t,q_t^\eps)}\right)\dot w(t),\\
q^{\eps,\varphi}_0=q^\eps_0=x_0.
\end{cases}
\end{equation*}
Therefore, one has
\begin{equation}\label{eq-qphiq}
\begin{aligned}
\abs{q^{\eps,\varphi}_t-q^\eps_t}\leq C\Big(\eps^2\|\varphi\|+1\Big)\|q^\eps-\varphi\|+|D_3^\varphi(t)|,
\end{aligned}
\end{equation}
where
$$
D_3^\varphi(t):=\sqrt{\eps}\int_0^t\left(\dfrac{\sigma(s,\eps^2\varphi_s)}{\lambda(s,\eps^2\varphi_s)}-\dfrac{\sigma(s,0)}{\lambda(s,0)}\right)dw(s).
$$

\noindent \textbf{Step 4:
Estimates of $\|X^\eps-\varphi\|$.} Applying \eqref{eq-XXphi}, \eqref{eq-Xphiqphi}, and \eqref{eq-qphiq} to \eqref{est}, we have
\begin{equation}\label{eq-X-phi-0}
\barray\disp
\sup_{r\in[0,t]}\abs{X^\eps_r-\varphi_r}\ad\leq
\wdt C\eps^2(1+\|\varphi\|+\|\varphi\|^2)
+\wdt C\Big(\eps^2\|\varphi\|+1\Big)\|q^\eps-\varphi\|+\wdt C\Big[1+\eps^2\int_0^1 |\dot\varphi_s|ds+\|H_\eps\|\Big]\|H_\eps^\varphi\|\\
\aad\ +\wdt C\eps^2(1+\sqrt\eps\|w\|)\left(1+\|H_\eps\|\right)^2(1+\|H_\eps^\varphi\|)\sup_{s\in[0,t]}|X^\eps_s-\varphi_s|\\
\aad \ +\wdt C(1+\|\varphi\|) \left[(1+\sqrt\eps\|w\|)(1+\|H_\eps\|)+\eps^2\|H_\eps^\varphi\|\right]\int_0^t\sup_{r\in[0,s]}|X^\eps_r-\varphi_r|ds\\
\aad \
+\wdt C\sup_{r\in[0,1]}\left(|D_1^\varphi(r)|+|D^\varphi_2(r)|+|D_3^\varphi(r)|\right),
\earray
\end{equation}
for a positive finite constant $\wdt C$,
independent of $\eps$  and $\varphi$.

\

\noindent\textbf{Final Step.}
To proceed, we need a couple of lemmas. To avoid interruption, the proofs of these lemmas are relegated to the appendix.

\begin{lm}\label{lm-Hphi}
	There are constants $\bar M_1$ and $\bar M_2$ independent of $\ell$ and $\eps$ such that
\begin{equation}\label{eq-Hphi}
\PP\left\{\sup_{t\in[0,1]}|H_\eps^\varphi(t)|>\ell\right\}\leq \bar M_1\exp\left\{-\frac{\bar M_2\ell^2}{\eps^2}\right\},\text{ for all }t,s\in [0,1], \ 0<\eps<1, \ell>0.
\end{equation}
\end{lm}

\begin{lm}\label{lm-Dphi}
	There is a constant $\bar M_3$ independent of $\eps$ such that
	$$ \PP\left\{\sup_{t\in[0,1]}\Big(|D_1^\varphi(t)|+|D_2^\varphi(t)|+|D_3^\varphi(t)|\Big)\geq \eps+\bar M_3\eps^2(\|X^\eps\|+\|\varphi\|)^2\right\}\leq \exp\left\{\frac{-1}{\eps^2}\right\}.
	$$
\end{lm}

With the two lemmas at hand, we proceed to complete the proof of the proposition.
Now, let $\theta, N>0$ be arbitrary and fixed.
By the LDP for the Brownian motion $w(\cdot)$ and Proposition \ref{prop-Heps}, there exists a
constant $L=L(N)>0$ and $\eps_1=\eps_1(L)\in (0,1)$ such that
\begin{equation}\label{eq-f-1}
\PP(\Omega^1_\eps)\leq \exp\left\{-\dfrac{3N}{\eps}\right\}, \quad \Omega^1_\eps:=\{\sqrt\eps\norm w+\|H_\eps\|+\|X^\eps\|>L\}\text{ for all }\eps<\eps_1.
\end{equation}
In view of Lemma \ref{lm-Dphi}, there is an $\eps_2=\eps_2(N,\bar M_3)\in (0,1)$ satisfying
\begin{equation}\label{eq-f-2}
\PP(\Omega^2_\eps)\leq \exp\left\{-\dfrac{3N}{\eps}\right\}, \quad \Omega^2_\eps:=\{\|D_1^\varphi\|+\|D_2^\varphi\|+\|D_3^\varphi\|>\eps+\bar M_3\eps^2(\|X^\eps\|+\|\varphi\|)^2\}\text{ for all }\eps<\eps_2.
\end{equation}
There is a small $\ell=\ell(N,\theta,\varphi)\in(0,1)$ satisfying
\begin{equation}\label{eq-f-3}
	\wdt C\ell\Big[1+\int_0^1 |\dot\varphi_s|ds+L\Big]e^{2\wdt C(1+\|\varphi\|)[(1+L)^2+1]}\leq \frac{\theta}{8}.
\end{equation}
By Lemma \ref{lm-Hphi}, there is an  $\eps_3=\eps_3(\ell,\bar M_1,\bar M_2,N)\in (0,1)$ such that
\begin{equation}\label{eq-f-4}
\PP(\Omega^3_\eps)\leq \exp\left\{-\dfrac{3N}{\eps}\right\}, \quad \Omega^3_\eps:=\{\|H_\eps^\varphi\|>\ell\}\text{ for all }\eps<\eps_3.
\end{equation}
There is an $\eps_4=\eps_4(N,\theta,\varphi)\in (0,1)$ such that for all $\eps<\eps_4$
\beq\label{eq-f-5}
\begin{aligned}
&\wdt C\eps^2(1+L)^3<\frac 14,\;
\wdt C\eps(2+\|\varphi\|+\|\varphi\|^2)e^{2\wdt C(1+\|\varphi\|)[(1+L)^2+1]}\leq \frac{\theta}8
\text{ and }\\
&\wdt C\eps[1+\bar M_3(L+\|\varphi\|)^2]e^{2\wdt C(1+\|\varphi\|)[(1+L)^2+1]}\leq \frac{\theta}8.
\end{aligned}
\eeq
Let
\begin{equation}\label{eq-5}
\bar \theta_1=\dfrac{\theta}{8\wdt C\Big(\eps^2\|\varphi\|+1\Big)\exp\{2\wdt C(1+\|\varphi\|)[(1+L)^2+1\}},
\end{equation}
and
$$
\Omega^0_\eps=\left\{\|q^\eps-\varphi\|<\bar \theta_1\right\}\setminus\left(\cup_{i=1}^3\Omega^i_\eps\right).
$$
Then it is clear
that $\bar\theta_1$ is independent of $\eps$. Moreover,
 note that
$$
\PP\left(\cup_{i=1}^3\Omega^i_\eps\right)\leq 3\exp\left\{-\frac{3N}{\eps}\right\}\leq \exp\left\{-\frac{N}{\eps}\right\}
$$
for all $\eps<\eps_5$ for some $\eps_5=\eps_5(N)\in (0,1)$.

Now, for any $\eps<\eps_0:=\min\{\eps_i:i=1,\dots,5\}$ and $\omega\in\Omega^0_\eps$, we have from \eqref{eq-f-1}, \eqref{eq-f-2}, and \eqref{eq-f-4} that
\begin{equation}\label{eq-X-phi-01}
\barray\disp
\sup_{s\in[0,t]}\abs{X^\eps_s-\varphi_s}\ad\leq
2\wdt C\eps(2+\|\varphi\|+\|\varphi\|^2)
+2\wdt C\Big(\eps^2\|\varphi\|+1\Big)\bar \theta_1+\wdt C\ell\Big[1+\int_0^1 |\dot\varphi_s|ds+L\Big]\\
\aad \ +2\wdt C\eps[1+\bar M_3(L+\|\varphi\|)^2]
+2\wdt C\left(1+\|\varphi\|\right)[(1+L)^2+1]\int_0^t\sup_{r\in[0,s]}|X^\eps_r-\varphi_r|ds.
\earray
\end{equation}
Applying Gronwall's inequality
to \eqref{eq-X-phi-01} and then using \eqref{eq-f-3}, \eqref{eq-f-5}, and \eqref{eq-5}, we obtain that
$
\|X^\eps-\varphi\|<\theta,
$
for any $\eps<\eps_0$ and $\omega\in\Omega^0_\eps$.
Therefore, one has that for any $\eps<\eps_0$,
$$
\PP\left\{\|q^\eps-\varphi\|<\theta\right\}\geq \PP\left\{\|X^\eps-\varphi\|<\bar\theta_1\right\}-\exp\left\{-\dfrac{N}{\eps}\right\}.
$$
By noting that
$$\abs{q^\eps_t-\varphi_t}\leq\abs{X^\eps_t-\varphi_t}+ \abs{X^\eps_t-X^{\eps,\varphi}_t}+\abs{X^{\eps,\varphi}_t-q^{\eps,\varphi}_t}
+\abs{q^{\eps,\varphi}_t-q^\eps_t},$$
\eqref{Cl-2}
can be
 obtained.
Therefore,
Proposition \ref{close}
is proved.
\end{proof}

Applying Assumption \ref{asp-2} and Proposition \ref{close} enables us to obtain the following theorem.

\begin{thm}\label{thm3.2}
	The sequence $\{X^\eps\}_{\eps>0}$ satisfies the local LDP.
That is,
for any
$\varphi\in\CC([0,1],\R^d)$, one has
$$
\begin{aligned}
\lim_{\theta\to 0}&\limsup_{\eps\to 0}\eps\log\PP\left(X^\eps\in B(\varphi,\theta)\right)\\
&=\lim_{\theta\to 0}\liminf_{\eps\to 0}\eps\log\PP\left(X^\eps\in B(\varphi,\theta)\right)\\
&=-\hat I(\varphi),
\end{aligned}
$$
where $B(\varphi,\eps)$ is the ball centered at $\varphi$ with radius $\eps$.
\end{thm}

\begin{proof} We divide the proof by treating the lower bounds and upper bounds.

\para{Lower bound of local LDPs.} We first
prove
\bea \ad
\lim_{\theta\to 0}\liminf_{\eps\to 0}\eps\log\PP\left(X^\eps\in B(\varphi,\theta)\right) \geq-\hat I(\varphi).
\eea
Since it is trivial if $\hat I(\varphi)=\infty$, we assume that $\hat I(\varphi)<\infty$ and $\varphi$ is absolutely continuous.
For any $r>0$, since $\{q^\eps\}_{\eps>0}$ satisfies the local LDP with rate function $\hat I$, there is a $\theta_1$ such that
$$
\liminf_{\eps\to 0}\eps\log\PP\left(q^\eps\in B(\varphi,\theta_1)\right)\geq -\hat I(\varphi)+2r.
$$
Let $N_1$ be sufficiently large such that
$$
\exp\left\{\frac{-\hat I(\varphi) +2r}{\eps}\right\}-\exp\left\{-\frac {N_1}\eps\right\}\geq \exp\left\{\frac{-\hat I(\varphi)+r}{\eps}\right\}.
$$
By Proposition \ref{close}, there are $\bar\theta_1$ and $\eps_0$ such that for any $\eps<\eps_0$
$$
\PP\left\{\|X^\eps-\varphi\|<\bar\theta_1\right\}\geq
\PP\left\{\|q^\eps-\varphi\|<\theta_1\right\}-\exp\left\{-\dfrac{N_1}{\eps}\right\}.
$$
As a consequence, one concludes that for any $r>0$, there is a $\bar\theta_1$ satisfying
$$
\begin{aligned}
\liminf_{\eps\to 0}\eps\log\PP\left(X^\eps\in B(\varphi,\bar \theta_1)\right)&\geq \liminf_{\eps\to 0}\eps\log\left(\PP\left(q^\eps\in B(\varphi,\theta_1)\right)-\exp\left\{-\frac {N_1}{\eps}\right\}\right)\\
&\geq -\hat I(\varphi)+r.
\end{aligned}
$$
Therefore, we obtain the lower bound for local LDPs.

\para
{Upper bound of local LDPs.}
It is easily seen that if $\varphi$ is absolutely continuous,
a similar argument to the process of obtaining lower bound of local LDP yields that
\bea \ad
\lim_{\theta\to 0}\limsup_{\eps\to 0}\eps\log\PP\left(X^\eps\in B(\varphi,\theta)\right)\\
\aad \ \leq-\hat I(\varphi).
\eea
Now, we consider $\varphi\in \CC([0,1],\R^d)$, which is not absolutely continuous and $I(\varphi)=\infty$. We aim to prove that
$$
\lim_{\theta\to 0}\limsup_{\eps\to 0}\eps\log\PP\left(X^\eps\in B(\varphi,\theta)\right)
=-\infty.
$$
For any $R>0$, since $\{q^\eps\}_{\eps>0}$ satisfies the local LDP with rate function $\hat I$ and $\hat I(\varphi)=\infty$, there is a $\theta_2\in (0,1)$ such that
$$
\limsup_{\eps\to 0}\eps\log\PP\left(q^\eps\in B(\varphi,\theta_2)\right)\leq -R.
$$
Let $N_2>R$.
By Proposition \ref{close} and \eqref{eq-5}, there is a $\bar\theta_2\in (0,\theta_2/2)$ such that for any $\phi\in B(\varphi,1)$, $\phi$ is absolutely continuous and there is $\eps_0=\eps_0(\phi)$ satisfying
$$
\PP\left\{\|X^\eps-\phi\|<2\bar\theta_2\right\}\leq
\PP\left\{\|q^\eps-\phi\|<\frac{\theta_2}2\right\}+\exp\left\{-\dfrac{N_2}{\eps}\right\},\;\forall \eps<\eps_0(\phi).
$$
Let $\bar\phi\in B(\varphi,\bar\theta_2)$ be an absolutely continuous function (such $\bar\phi$ does always exist due to denseness of absolutely continuous functions).
As a consequence, we have
$$
\begin{aligned}
\limsup_{\eps\to 0}\eps\log\PP\left(X^\eps\in B(\varphi,\bar \theta_2)\right)&\leq \limsup_{\eps\to 0}\eps\log\PP\left(X^\eps\in B(\bar\phi,2\bar \theta_2)\right)\\
&\leq \limsup_{\eps\to 0}\eps\log\left(\PP\left(q^\eps\in B(\bar\phi,\frac{\theta_2}2)\right)+\exp\left\{-\frac {N_2}{\eps}\right\}\right)\\
&\leq \limsup_{\eps\to 0}\eps\log\left(\PP\left(q^\eps\in B(\varphi,\theta_2)\right)+\exp\left\{-\frac {N_2}{\eps}\right\}\right)\\
&=\max\left\{\limsup_{\eps\to 0}\eps\log\PP\left(q^\eps\in B(\varphi,\theta_2)\right),-N_2\right\}\\
&\leq-R.
\end{aligned}
$$
Therefore, if $\varphi$ is not absolutely continuous, then
\bea \ad
\lim_{\theta\to 0}\limsup_{\eps\to 0}\eps\log\PP\left(X^\eps\in B(\varphi,\theta)\right)
 =-\infty.
\eea
So, the proof is complete.
\end{proof}

\appendix
\section{Proofs of Technical Results}

\begin{proof}[Proof of Proposition \ref{prop-Heps}]
	If $f\in \CC^1([0,t])$ and $g\in \CC([0,t])$, then the Stiltjies integral
	$$\int_0^t f(s)dg(s),\;\;\;\;\;t\geq 0$$
	is well defined and the following integration by parts formula holds
	\begin{equation}\label{bypart1}
	\int_{t_1}^{t_2} f(s)dg(s)=f(t_2)g(t_2)-f(t_1)g(t_1)-\int_{t_1}^{t_2} g(s)f'(s)ds,\;\;\;0\leq t_1<t_2\leq t.
	\end{equation}
	In addition, if $g(0)=0$, as a consequence of \eqref{bypart1},
	\begin{equation}\label{bypart}
	\int_{0}^{t} f(s)dg(s)=g(t)f(0)+\int_{0}^{t} \left(g(t)-g(s)\right)f'(s)ds,\quad t\geq 0.
	\end{equation}
	Thus, we can apply the integration by parts formula \eqref{bypart} for $f(s)=e^{A_\eps(s)}\sigma(s,\eps^2X^\eps_s)$ and $g(s)=w(s)$
	to get
	\begin{equation}\label{divide}
	\begin{aligned}
	\int_0^t &e^{A_\eps(s)}\sigma_\eps(s,X^\eps_s)dw(s)=\sigma_\eps(0,x_0)w(t)
	\\&\;\;+\int_0^t e^{A_\eps(s)}\left[\dfrac{\lambda_\eps(s,X^\eps_s)}{\eps^2}\sigma_\eps(s,X^\eps_s)+\nabla_s\sigma_\eps(s,X^\eps_s)+\nabla_x\sigma_\eps(s,X^\eps_s)p^\eps_s\right]\left(w(t)-w(s)\right)ds.
	\end{aligned}
	\end{equation}
	Therefore, by multiplying $\sqrt\eps e^{-A_\eps(t)}$ to both sides of \eqref{divide},
	taking the norm on both sides
	of the equation, using boundedness assumptions on $\lambda_\eps$ and $\sigma_\eps$,
	and  carrying out the detailed calculations,
	we obtain
	\begin{equation}\label{2.7}
	\abs{H_\eps(t)} \leq C\sqrt{\eps} \norm {w}\left(1+\frac 1{\eps^2}\int_0^t e^{-A_\eps(t,s)}ds+\eps^2\int_0^te^{-A_\eps(t,s)}|p^\eps_s|ds\right).
	\end{equation}
	Combining \eqref{2.7} and the fact $A_\eps(t,s)\geq \dfrac{\kappa_0 (t-s)}{\eps^2}\geq 0$ implies that
	\beq\label{eq-H}
	\abs{H_\eps(t)} \leq C\sqrt{\eps} \norm {w}\left(1+\eps^2\int_0^t|p^\eps_s|ds\right),\;\forall t\in [0,1].
	\eeq
	
	Now, we are ready to estimate $|p^\eps_s|$.
	Thanks to \eqref{formulap} and the fact $A_\eps(t,s)\geq \dfrac{\kappa_0 (t-s)}{\eps^2}$ again, we have
	\begin{equation}\label{2.8}
	\abs{p^\eps_t}\leq\abs{x_1}e^{-A_\eps(t)}+ \frac C{\eps^2}\int_0^t e^{-\frac{\kappa_0 (t-s)}{\eps^2}}\left(1+\abs{X^\eps_s}\right)ds+\dfrac 1{\eps^2}\abs{H_\eps(t)}.
	\end{equation}
	Because of \eqref{qint-1} and Young's inequality, one has
	$$
	\begin{aligned}
	|X^\eps_t|&\leq C+C\int_0^t (1+|X^\eps_s|)ds+\frac 1{\eps^2}\int_0^t |H_\eps(s)|ds\\
	&\leq C\left(1+\frac 1{\eps^2}\int_0^t|H_\eps(s)|ds\right)+C\int_0^t |X^\eps_s|ds.
	\end{aligned}
	$$
	Then applying Gronwall's inequality leads to
	\begin{equation}\label{eq-xh-1}
	|X^\eps_t|\leq C\left(1+\frac 1{\eps^2}\int_0^t|H_\eps(s)|ds\right).
	\end{equation}
	Hence, we obtain from \eqref{2.8} and \eqref{eq-xh-1}
	that
	\begin{equation}\label{eq-2.8-1}
	\abs{p_\eps(t)}\leq  C\left(1+\frac 1{\eps^2}\sup_{s\in[0,t]}|H_\eps(s)|\right).
	\end{equation}
	Applying \eqref{eq-2.8-1} to \eqref{eq-H}
	yields
	that
	$$
	\sup_{s\in[0,t]}\abs{H_\eps(s)} \leq C\sqrt{\eps} \norm {w}\left(1+\int_0^t\sup_{r\in[0,s]}|H_\eps(r)|ds\right),\;\forall t\in [0,1].
	$$
	Thus, Gronwall's inequality
	implies that
	$$
	\|H_\eps\|\leq \hat C_0\sqrt{\eps} \norm {w}e^{\hat C_0\sqrt{\eps} \norm {w}},
	$$
	for some finite constant $\hat C_0$, independent of $\eps$.
	
	Next, we proceed
	to estimate $X^\eps$ by using the representation \eqref{qint}.
	It is easily seen that for all $t\in [0,1]$,
	\begin{equation}\label{eq-R-1-2-4}
	|R_\eps^{(1)}(t)|\leq C,\quad
	|R_\eps^{(2)}(t)|\leq C\left(1+\int_0^t\sup_{r\in [0,s]}|X^\eps_r|ds\right),\quad
	|R_\eps^{(4)}(t)|\leq C|H_\eps(t)|.
	\end{equation}
	By assumptions on bounded derivative of $\lambda$ and \eqref{eq-2.8-1}, one can obtain that
	\beq
	|R_\eps^{(3)}(t)|\leq C(1+\|H_\eps\|)\int_0^t \sup_{r\in [0,s]}|X^\eps_r|ds.
	\eeq
	Similarly, one has
	\beq
	|R_\eps^{(5)}(t)|\leq C\|H_\eps\|(1+\|H_\eps\|).
	\eeq
	Therefore, we have that
	\beq\label{eq-Reps}
	\begin{aligned}
		|R_\eps(t)|&\leq C(1+\|H_\eps\|+\|H_\eps\|^2)+C(1+\|H_\eps\|)\int_0^t \sup_{r\in [0,s]}|X^\eps_r|ds.
	\end{aligned}
	\eeq
	Moreover, a similar process of getting \eqref{divide} helps us
	to estimate
	$$
	\sqrt\eps\int_0^t \frac{\sigma(s,\eps^2X^\eps_s)}{\lambda\left(s,\eps^2X^\eps_s\right)}dw(s),
	$$
	which together with \eqref{eq-2.8-1}
	implies
	that
	\beq\label{eq-difX}
	\left|\sqrt\eps\int_0^t \frac{\sigma(s,\eps^2X^\eps_s)}{\lambda\left(s,\eps^2X^\eps_s\right)}dw(s)\right|\leq C\sqrt\eps\|w\|(1+\|H_\eps\|).
	\eeq
	It follows from
	\eqref{qint}, \eqref{eq-Reps}, and \eqref{eq-difX} that
	\begin{equation}\label{eq-29}
	\sup_{s\in[0,t]}|X^\eps_s|\leq C(1+\|H_\eps\|+\|H_\eps\|^2)(1+\sqrt\eps\|w\|)+C(1+\|H_\eps\|)\int_0^t \sup_{r\in[0,s]}|X^\eps_r|ds.
	\end{equation}
	Therefore, it follows \eqref{eq-29} and the Gronwall inequality that
	\begin{equation}\label{eq-Xeps}
	\|X^\eps\|\leq C(1+\|H_\eps\|+\|H_\eps\|^2)(1+\sqrt\eps\|w\|)e^{C(1+\|H_\eps\|)}\leq \hat C_1\Gamma(\hat C_1\sqrt\eps\norm{w}),
	\end{equation}
	where
	$\hat C_1$ is a positive finite constant, independent of $\eps$ and
	$$
	\Gamma(v):=(1+ve^v+v^2e^{2v})(1+v)e^{1+ve^v},\;v\geq 0.
	$$
	The proof is complete.
\end{proof}

\begin{proof}[Proof of Lemma \ref{lm-1111}]
	The proof is similar to \cite[Proof of Theorem 4.2]{DKMN09}.
	Define the grid $\G_n=\{\frac i{2^{n}}:0\leq i\leq 2^{n}\}$. Two points $u=\frac i{2^{n}}, v=\frac j{2^{n}}\in\G_n$ are said to be nearest neighbor if $|i-j|\leq 1$.
	Then for any $u\in\G_n$, there exists a path $0 = q_0, v_1,\dots, v_N = u$ of points in
	$\G_n$ such that each pair $v_{i-1}$ and $v_i$ are nearest neighbors in some grid $\G_m$, $m \leq
	n$, and at most one of such pairs consists of points, which are nearest
	neighbors in any given grid $\G_m$. Indeed, we can write
	$u = 0.k_1\dots k_N$ in the binary (base 2) expansion and let $v_m = 0.k_1k_2\dots k_m$.
	Next, let $\mathcal D(n)$ be the event that for all nearest
	neighbors  $u,v\in\G_n$, we have
	$|Y(u) - Y(v)| \leq L2^{-0.125n}$.
	From \eqref{eq-1111},
	for each pair of nearest neighbors $u,v\in\G_n$, we have
	$$
	\PP (|Y(u) - Y(v)|> L2^{-0.125n}) \leq\alpha_1 \exp\left\{-\alpha_22^{0.25n} \right\}.
	$$
	Because there are $2^n$ nearest neighbors in $\G_n$, one gets
	$$
	\PP((\mathcal D(n))^c) \leq 2^{n}\alpha_1\exp\left\{-\alpha_22^{0.25n}\right\}\leq
	C_1\alpha_1\exp\left\{-C_2\alpha_22^{0.25n}\right\},
	$$
	for some positive constants $C_1,C_2$, independent of $n$.
	Hence, let $\mathcal D= \cap_{n=0}^\infty \mathcal D(n)$ and summing the previous estimates over $n$, we have
	$$\PP(\mathcal D^c) \leq C_1\alpha_1\exp\left\{-C_2\alpha_2\right\},$$
	where $C_1,C_2$ may be different than before.
	Moreover, in the event $\mathcal D$ one has that for any $u\in\cup_{n=0}^\infty\G_n$, there is a path $0=v_1;v_2;\dots;v_N=u$ with $v_{i-1},v_i$ are nearest neighbors in some $\G_n$ and then,
	$$
	|Y(u)|\leq \sum_{n=1}^N|Y(v_{n-1})-Y(v_n)|\leq\sum_{n=1}^\infty L2^{-0.125n}\leq C_3L.
	$$
	Therefore, we
	conclude the proof of Lemma \ref{lm-1111}.
\end{proof}

\begin{proof}[Proofs of Lemmas {\rm\ref{lm-Hphi}} and {\rm\ref{lm-Dphi}}]
	A standard calculation shows that
	\begin{equation}
	\begin{aligned}
	H_\eps^\varphi(t)-H_\eps^\varphi(s)=&\sqrt{\eps} e^{-A_\eps^\varphi(t)}\int_0^t e^{A^\varphi_\eps(r)}\sigma_\eps(r,\varphi_r)dw(r)-\sqrt{\eps} e^{-A_\eps^\varphi(s)}\int_0^s e^{A^\varphi_\eps(r)}\sigma_\eps(r,\varphi_r)dw(r)\\
	=&\sqrt{\eps} \int_s^t e^{-A^\varphi_\eps(t,r)}\sigma_\eps(r,\varphi_r)dw(r)-\sqrt{\eps}(1-e^{-A_\eps^\varphi(t,s)}) \int_0^s e^{-A^\varphi_\eps(s,r)}\sigma_\eps(r,\varphi_r)dw(r).
	\end{aligned}
	\end{equation}
	As used often in this paper, the first stochastic integral is an element of a sequence of martingales with quadratic deviation bounded by
	$C\eps^3(1-e^{-\frac{\kappa_0(t-s)}{\eps^2}})$, then using the fact $1-e^{-u}\leq\sqrt u,\forall u>0$ that is bounded
	by
	$
	C\eps^2\sqrt{|t-s|}.
	$
	Similarly,
	by using the fact $(1-e^{-u})^2\leq\sqrt u,\forall u>0$, the second stochastic integral is an element of a sequence of martingales with quadratic deviation bounded by the
	$C\eps^2\sqrt{|t-s|}.$
	Therefore, an application of exponential martingale inequality  \cite[Theorem 7.4, p. 44]{Mao97} allows us to obtain that $\forall t,s\in [0,1]$,
	\begin{equation}\label{eq-Hphi-ts}
	\PP\left\{|H_\eps^\varphi(t)-H_\eps^\varphi(s)|>\ell\right\}\leq \exp\left\{-\frac{C\ell^2}{\eps^2|t-s|^{\frac 12}}\right\},
	\end{equation}
	where $C$ is some finite constant, independent of $\ell,\eps$.
	With this property, the technique and argument to obtain \eqref{eq-Hphi} is similar to that of Lemma \ref{lm-1111}.
	Similarly, the proof of Lemma \ref{lm-Dphi} is obtained
	by using exponential martingale inequality  \cite[Theorem 7.4, p. 44]{Mao97}.
\end{proof}

\

\

\para{Data Availability Statements.}
Data sharing is not applicable to this article as no new data were created or analyzed in this study.

\end{document}